\documentclass[reqno,11pt]{amsart}
\usepackage{amscd,amssymb,verbatim}
\usepackage{enumitem}
\usepackage{marginnote}
\usepackage[nobysame]{amsrefs}
\usepackage{hyperref}
\usepackage[scr=rsfso]{mathalfa}
\usepackage{mathrsfs}  \usepackage{bbm}
\usepackage{xcolor}
\usepackage{slashed}
\usepackage{graphicx}

\DeclareFontFamily{U}{BOONDOX-calo}{\skewchar\font=45 }
\DeclareFontShape{U}{BOONDOX-calo}{m}{n}{
  <-> s*[1.05] BOONDOX-r-calo}{}
\DeclareFontShape{U}{BOONDOX-calo}{b}{n}{
  <-> s*[1.05] BOONDOX-b-calo}{}
\DeclareMathAlphabet{\mathcalboondox}{U}{BOONDOX-calo}{m}{n}
\SetMathAlphabet{\mathcalboondox}{bold}{U}{BOONDOX-calo}{b}{n}
\DeclareMathAlphabet{\mathbcalboondox}{U}{BOONDOX-calo}{b}{n}

\numberwithin{equation}{section}

\theoremstyle{plain}
\newtheorem{theorem}{Theorem}[section]
\newtheorem{lemma}[theorem]{Lemma}
\newtheorem{corollary}[theorem]{Corollary}
\newtheorem{proposition}[theorem]{Proposition}

\newtheorem{assumption}[theorem]{Assumption}

\theoremstyle{definition}
\newtheorem{definition}[theorem]{Definition}

\theoremstyle{remark}
\newtheorem{remark}[theorem]{Remark}

\newcommand{\IM}{\mathrm{Im}}
\newcommand{\RR}{\mathbb{R}}
\newcommand{\CC}{\mathbb{C}}

\newcommand{\ZZ}{\mathbb{Z}}
\renewcommand{\>}{\rangle}
\newcommand{\<}{\langle}

\newcommand{\ind}{\operatorname{ind}}

\newcommand{\End}{\operatorname{End}}

\newcommand{\p}{\partial}

\newcommand{\calB}{\mathscr{B}}

\newcommand{\calA}{\mathscr{A}}

\newcommand{\calW}{\mathscr{W}}\newcommand{\calH}{\mathscr{H}}
\newcommand{\calWp}{\mathscr{W}_{[-\pi,\pi]}}\newcommand{\calHp}{\mathscr{H}_{[-\pi,\pi]}}

\newcommand{\At}{\calA_\tau}
\newcommand{\calS}{\mathscr{S}}
\newcommand{\HH}{\mathcal{H}}

\newcommand{\D}{\mathscr{D}}\newcommand{\B}{\mathscr{B}}
\newcommand{\M}{\mathscr{M}}\newcommand{\N}{\mathscr{N}}
\newcommand{\E}{\mathscr{E}}\newcommand{\F}{\mathscr{F}}
\newcommand{\MM}{\mathscr{M}}

\renewcommand{\c}{\mathcalboondox{c}}
\newcommand{\spf}{\operatorname{sf}}
\newcommand{\sign}{\operatorname{sign}}
\newcommand{\rank}{\operatorname{rank}}
\newcommand{\ID}{\operatorname{Id}}
\newcommand{\Herm}{\operatorname{Herm}}

\newcommand{\f}{\mathcalboondox{f}}

\newcommand{\oM}{\overline{M}}

\newcommand{\Ib}{\mathcal{I}_{{\rm Bulk},\tau}}
\newcommand{\Ie}{\mathcal{I}_{{\rm Edge},\tau}}
\renewcommand{\b}{*}
\newcommand{\odd}{{\operatorname{odd}}}
\newcommand{\even}{{\operatorname{even}}}

\newcommand{\n}{\nabla}

\begin{document}

\title{The $\ZZ_2$-valued spectral flow of a symmetric family of Toeplitz operators}
\author{Maxim Braverman} 
\address{Department of Mathematics,
        Northeastern University,
        Boston, MA 02115,
        USA
         }
\email{maximbraverman@gmail.com}
\author{Ahmad Reza Haj Saeedi Sadegh }
    \address{Department of Mathematics, Dartmouth College, Hanover, NH 03755,
USA}
\email{arhsaeedi@gmail.com}

\begin{abstract}
We consider families $A(t)$ of self-adjoint operators with symmetry that causes the spectral flow of the family to vanish. We study the secondary $\mathbb{Z}_2$-valued spectral flow of such families. We prove an analog of the Atiyah-Patodi-Singer-Robbin-Salamon theorem, showing that this secondary spectral flow of $A(t)$ is equal to the secondary $\mathbb{Z}_2$-valued index of the suspension operator $\frac{d}{dt}+A(t)$.

Applying this result, we show that the graded secondary spectral flow of a symmetric family of Toeplitz operators on a complete Riemannian manifold equals the secondary index of a certain Callias-type operator. In the case of a pseudo-convex domain, this leads to an odd version of the secondary Boutet de Monvel's index theorem for Toeplitz operators. When this domain is simply a unit disc in the complex plane, we recover the bulk-edge correspondence for the Graf-Porta module for 2D topological insulators of type AII.
\end{abstract}


\maketitle
\section{Introduction}\label{S:introduction}

Atiyah, Patodi, and Singer \cite{APS3} showed that the spectral flow of a periodic path $A=A(t)$ of self-adjoint Fredholm operators on a Hilbert space $H$ is equal to the index of the ``suspension" operator $D_A:= \frac{d}{dt}+A(t)$. This result was extended to non-periodic families by Robbin and Salamon \cite{RobbinSalamon95}. 

In this paper, we assume that there is an anti-linear, anti-unitary, anti-involution $\alpha: H\to H$, $\alpha^2=-1$, $\alpha^*=\alpha$, and define a transformation $\tau$ of paths of operators by  $\tau: A(t)\to \alpha A(-t) \alpha^{-1}$. If the path $A(t)$ is $\tau$-invariant, then its spectral flow vanishes.  However, the secondary $\ZZ_2$-valued invariant called the {\em half-spectral flow} is defined, cf. \cites{DollSB21,DeNittisSB15}. In this situation, the suspension operator $D_A$ commutes with $\tau$. Such an operator is called {\em odd symmetric} and its index vanishes because of the symmetry. But a secondary $\ZZ_2$-valued index $\ind_\tau D_A$, called the {\em $\tau$-index},  is defined, \cites{Schulz-Baldes15,DeNittisSB15,DollSB21,
BrSaeedi24index}.

The first main result of our paper is the analog of the Atiyah-Patodi-Singer-Robbin-Salamon theorem for half-spectral flow: 
\begin{equation}\label{E:Ispf=indDA}
        \spf_\tau A \ = \ \ind_\tau D_A.
\end{equation}
Note that recently a different version of the  Robin-Salamon-type theorem for the $\ZZ_2$-valued spectral flow was obtained in \cite{BourneCareyLeschRennie22}. In this paper, the authors considered paths of skew-adjoint Fredholm operators on a real vector space. The main difference with our result is that in \cite{BourneCareyLeschRennie22} every operator $A(t)$ possesses a symmetry (it is skew-adjoint), while we consider the case when the symmetry relates $A(t)$ with $A(-t)$. We discuss the relationship between our result and \cite{BourneCareyLeschRennie22} in Section~\ref{SS:compareBCLR}.

Next, we consider a complete Riemannian manifold $M$ with involution $\theta^M: M\to M$. Let $E=E^+\oplus E^-$ be a graded Clifford bundle over $M$ endowed with an anti-unitary anti-involution $\theta^E$ covering $\theta^M$. We say that $E$ is a {\em graded quaternionic Clifford bundle}. Then the corresponding Dirac operator $D$  is odd symmetric.  We denote by $\HH=\HH^+\oplus \HH^-$ the kernel of $D$ and by $P$ the orthogonal projection onto $\HH$. Typically, it is an infinite dimensional space. 

Given a family $f_t$ of self-adjoint matrix-valued functions on $M$, parametrized by a circle $t\in S^1$, we denote by $M_{f_t}$ the operator of multiplication by $f_t$ and define a family of Toeplitz operators 
\[
    T_{f_t}\ := \ PM_{f_t}P:\,\HH\ \to \ \HH, \qquad t\in S^1.
\]
Let $T_{f_t}^\pm$ denote the restriction of $T_{f_t}$ to $\mathcal{H}^\pm$. Under suitable assumptions on $D$ and $f_t$, cf. Section~\ref{S:sfToeplitz}, the operators $T_{f_t}$ are Fredholm. We assume that the family $f_t$ is odd symmetric. Then so is the family $T_{f_t}$. 

The Dirac operator $D$ induces an ungraded Dirac operator $\D$ on the manifold $\M:=S^1\times M$. For large enough constant $c$, the operator 
\[
    \B_{cf}\ := \ \D\  + \ ic\M_{f_t}
\]
is an odd symmetric Callias-type operator. The $\tau$-index of such operators was investigated in \cite{BrSaeedi24index}. We show that the {\em graded half-spectral flow} 
\begin{equation}\label{E:IsfT=indB}
  \spf_\tau T_{f_t}^+ \ - \  \spf_\tau T_{f_t}^- \ = \ \ind_\tau \B_{cf}.
\end{equation}

In the case when $M$ is a pseudo-convex domain in $\CC^n$ with smooth boundary $N$, it follows from \cite[\S5]{DonnellyFefferman83} that $\HH^-=\{0\}$. Hence, the above formula computes the half-spectral flow of $T_{f_t}^+$. Applying the Callias index theorem for $\tau$-index, which was obtained in our previous paper \cite{BrSaeedi24index}, we obtain an explicit formula for a graded Dirac operator $\D_\N$ on $\N:=S^1\times N$, such that 
\begin{equation}\label{E:generalized bulk-edge}
        \spf_\tau T_{f_t}^+ \ = \  \ind_\tau \D_\N^+.
\end{equation}
(Note that $\spf_\tau T_{f_t}^-=0$ since $\HH^-=\{0\}$).
We call this result the {\em generalized bulk-edge correspondence} because if $M=\{z\in \CC:\, |z|\le1\}$ is the unit disc in $\CC$ this is equivalent to the bulk-edge correspondence for the Graf-Porta model for 2D topological insulators with time-reversal symmetry, \cite{GrafPorta13}. In Section~\ref{S:Graf-Porta}, we use \eqref{E:generalized bulk-edge} to obtain a new proof of the bulk-edge correspondence in the Graf-Porta model.

The paper is organized as follows:

In Section~\ref{S:tauindex}, we recall the main facts about the $\tau$-index of an odd symmetric operator from \cites{Schulz-Baldes15,DeNittisSB15,DollSB21,BrSaeedi24index}. 

In Section~\ref{S:spaces}, we recall the main steps of the Robbin-Salamon proof of the equality between the spectral flow of $A(t)$ and the index of $D_A$. The notation introduced in this section is used in the subsequent sections. 

In Section~\ref{S:tauabsractflow}, we study $\tau$-invariant paths $A(t)$ of self-adjoint Fredholm operators parametrized by $t\in \RR$. We define their half-spectral flow and show that it has many properties similar to the usual $\ZZ$-valued spectral flow. The main result of this section is the equality \eqref{E:Ispf=indDA} between the half-spectral flow of $A(t)$ and the $\tau$-index of $D_A$. The proofs of many statements in this section are postponed to Section~\ref{S:proofs of three theorems}. 

In Section~\ref{S:APScircle}, we show that the equality \eqref{E:Ispf=indDA} also holds for periodic paths, i.e. for paths parametrized by a circle. 

In Section~\ref{S:sfToeplitz}, we study families of Toeplitz operators on complete Riemannian manifolds $M$ associated with a $\tau$-invariant family of functions $f_t:M\to \Herm(k)$ with values in the space $\Herm(k)$ of Hermitian $k\times k$-matrices. We compute the graded spectral flow of such families, \eqref{E:Ispf=indDA}. The proof of the main result is postponed to Section~\ref{S:prsfToeplitz=Callias}.

In Section~\ref{S:evendim}, we consider the case when the dimension of $M$ is even and give a more explicit formula for the right-hand side of \eqref{E:Ispf=indDA} using the $\ZZ_2$-version of the Callias index theorem. When $M$ is a pseudo-convex domain in $\CC^n$ the formula is especially nice and we call it the {\em generalized bulk-edge correspondence}.

In Section~\ref{S:Graf-Porta}, we recall the Graf-Porta model for topological insulators of type AII and show that the result of the previous section implies the equality between the bulk and the edge index in this model. 

In Section~\ref{S:prsfToeplitz=Callias}, we prove the main result of Section~\ref{S:sfToeplitz}.

\section{The $\tau$-index}\label{S:tauindex}

In this section, we recall the construction of the secondary $\ZZ_2$-valued index of odd symmetric operators  \cites{Schulz-Baldes15,DeNittisSB15,DollSB21}.  We work with unbounded operators and recall some results about the $\ZZ_2$-valued index of an odd symmetric differential operator on a non-compact involutive manifold  \cite{BrSaeedi24index}.

\subsection{An anti-linear involution}\label{SS:involtiontau}
Let $H$ be a complex Hilbert space.  We denote the inner product on $H$ by $\<\cdot,\cdot\>_H$. An \textbf{anti-linear}  map  $\tau:H\to H$  is called an anti-involution if  $\tau^*=-\tau= \tau^{-1}$.  Here  we denote by $\tau^*$ the unique anti–linear operator satisfying 
\[
	\<\tau x, y\>_H\ = \ \overline{\<x,\tau^*y\>_H} \qquad\text{for all}\quad
	x,y\in H.
\]
We say that $\tau$ is an {\em anti-unitary anti-involution}.

The following version of  \textit{Kramers' degeneracy} \cite{KleinMartin1952}  is proven in \cite[Lemma~2.2]{BrSaeedi24index}

\begin{lemma}\label{L:even} An anti-linear anti-involution $\tau$ does not have non-zero fixed vectors, i.e. $\tau x=x$ iff $x=0$. Further, if $\tau$ acts on a finite-dimensional space $H$, then $\dim H$ is even.  Moreover, there exists a subspace $L\in H$ such that $H=L\oplus \tau L$.
\end{lemma}

\subsection{Odd symmetric operators}\label{SS:oddsymmetric}
Let $W\subset H$ be a dense subspace and let $D: H\to H$ be a closed unbounded linear operator on $H$ whose domain is $W$.  We endow  $W$ with the inner product
\[
	\<x,y\>_W\ = \ \<x,y\>_H\ + \ \<Dx,Dy\>_W.
\]
Then $W$ is a Hilbert space and the operator $D:W\to H$ is a bounded. 

\begin{definition}\label{D:odd symmetric operator}
A closed operator $D:H\to H$ with a dense domain $W$ is called {\em odd symmetric} (or $\tau$-symmetric) if 
\begin{equation}\label{E:tauDtau}
    \tau\, D\, \tau^{-1}\ = \ D^*.
\end{equation}
We denote by $\calB_\tau(W,H)$ the space of $\tau$-symmetric operators on $H$ whose domain is $W$. This is a Banach space with the norm defined as the operator norm of $D:W\to H$. 
\end{definition}

\subsection{Graded odd symmetric operators}\label{SS:graded oddsymmetric}
If $H=H^+\oplus H^-$ is a graded Hilbert space, we say that $D$ is \textit{odd with respect to the grading} if $D:H^\pm\to H^\mp$. We denote $D^\pm$ the restriction of $D$ to $H^\pm$ and write $D=\left(\begin{smallmatrix}
0&D^-\\D^+&0\end{smallmatrix}\right)$.

Let $W$ be the domain of $D$. Then $W^\pm:= W\cap H^\pm$ is the domain of $D^\pm$ and $W= W^+\oplus W^-$.  If, in addition,  $D$ is self-adjoint, then $(D^\pm)^*= D^\mp$.

Assume that the anti-unitary anti-involution $\tau$ is odd with respect to the grading,  $\tau:H^\pm\to H^\mp$. If $D$ is odd symmetric, then \eqref{E:tauDtau} implies that $\tau D^\pm\tau^{-1}= (D^\mp)^*$. In this situation, we say that $D$ is a \textit{graded odd symmetric operator}.

We denote the space of graded odd symmetric self-adjoint operators by $\widehat{\calB}_\tau(W,H)$.
If $D\in \widehat{\calB}_\tau(W,H)$, then $\tau D^\pm \tau^{-1}= D^\mp= (D^\pm)^*$. Then \textit{the operators $D^\pm$ are odd symmetric}. Notice, that in this case $\tau:W^\pm\to W^\mp$.

\subsection{The $\ZZ_2$-valued index}\label{SS:Z2index}
If $D$ is an odd symmetric operator then $\ind D=0$ by \eqref{E:tauDtau}. Following \cite{Schulz-Baldes15}, we define the {\em secondary} $\ZZ_2$-valued invariant – the \textit{$\tau$-index of $D$}:

\begin{definition}\label{D:tau index}
If $D$ is an odd symmetric operator on $H$ its {\em $\tau$-index} is
\begin{equation}\label{E:tauindex}
	\ind_\tau D\ := \ \dim \ker D \qquad \mod 2.
\end{equation}
If $D\in \widehat{\calB}_\tau(W,H)$ is a graded odd symmetric self-adjoint operator on a graded Hilbert space $H$, we define 
\begin{equation}\label{E:tauindex graded}
	\ind_\tau D^+\ := \ \dim \ker D^+ \qquad \mod 2.
\end{equation}
\end{definition}

The following version of the homotopy invariance of the index is proven in \cite[Theorem~2.5]{BrSaeedi24index}:
\begin{theorem}\label{T:homotopyindex}
The $\tau$-index $\ind_\tau D$ is constant on the connected components of $\calB_\tau(W,H)$. The $\tau$-index  $\ind_\tau D^+$ is constant on the connected components of $\widehat{\calB}_\tau(W,H)$.
\end{theorem}

We now discuss some geometric examples of odd symmetric operators, \cite{BrSaeedi24index}.

\subsection{Quaternionic vector bundles}\label{SS:quaternionic}
An\textit{ involutive manifold}  $(M,\theta^M)$ is  a Riemannian manifold $M$ together with a metric preserving involution  $\theta^M:M\to M$.  A {\em quaternionic vector bundle} over an involutive manifold $M$ is hermitian vector bundle $E$ endowed with an anti-linear unitary bundle map $\theta^E:E\to \theta^*E$, such that $(\theta^E)^2=-1$, cf. \cites{Dupont69,DeNittisGomi15,Hayashi17}. This means that for every $x\in M$ there exists an anti-linear map $\theta^E_x:E_x\to E_{\theta(x)}$, which depends smoothly on $x$ and satisfies $\theta^E_{\theta^M(x)}\, \theta^E_x=-1$,
$(\theta^E_x)^*\theta^E_x=1$.

If $E=E^+\oplus E^-$ is a graded vector bundle and  $\theta^E$ is {\em odd with respect to the grading}  (i.e. $\theta^E(E^\pm)=E^\mp$), we say that $(E,\theta^E)$ is a {\em graded quaternionic bundle}.

Given a quaternionic vector bundle $(E,\theta^E)$ over an involutive manifold  $(M,\theta^M)$ we define an anti-unitary anti-involution $\tau$ on the space $\Gamma(M,E)$ of sections of $E$ by 
\begin{equation}\label{E:tau=}
    \tau:\, f(x)\ \mapsto \ \theta^E f(\theta^M x), \qquad f\in\Gamma(M,E).  
\end{equation}
\subsection{Odd symmetric elliptic operators}\label{SS:oddelliptic}
Let $(E=E^+\oplus E^-,\theta^E)$ be a graded quaternionic vector bundle over an involutive manifold $(M,\theta^M)$. Let $D:\Gamma(M,E)\to \Gamma(M,E)$ be a self-adjoint odd symmetric elliptic differential operator on $M$, which is odd with respect to the grading, i.e. $D:\Gamma(M,E^\pm)\to \Gamma(M,E^\mp)$. We set $D^\pm:= D|_{\Gamma(M,E^\pm)}$. Then $\tau\, D^\pm\, \tau^{-1}= (D^\pm)^*= D^\mp$ and
\[
    D\ = \ \begin{pmatrix}
        0&D^-\\D^+&0
    \end{pmatrix}
\]
with respect to the decomposition $E=E^+\oplus E^-$.

Suppose now that $D$ is Fredholm. In particular, this is always the case when $M$ is compact.  Then the $\tau$-index $\ind_\tau D^+\in \ZZ_2$ is defined by \eqref{E:tauindex}. In \cite{BrSaeedi24index} we computed this index in several examples and proved the $\ZZ_2$-valued analogs of the Relative Index Theorem, the Callias index theorem, and the Boutet de Monvel's index theorem for Toeplitz operators.

\section{The spectral flow and the Robbin-Salamon theorem}\label{S:spaces}

This section recalls some basic definitions and results from \cite{RobbinSalamon95}. We also fix the notation which will be used in the next section.

\subsection{Some spaces of operators}\label{SS:nonequiv}
Let $A:H\to H$ be a closed operator  with dense domain $W$. As in Section~\ref{SS:oddsymmetric},  we define the scalar product on $W$ by 
\[
	\<x,y\>_W\ = \ \<x,y\>_H\ + \ \<Ax,Ay\>_H
\]
and view $A$ as a bounded operator  $A:W\to H$. We denote by $\|A\|$ its norm.  

We now make a basic 

\begin{assumption}\label{A:compact embedding}
The embedding $W\hookrightarrow H$ is compact.
\end{assumption}

Let $\calS(W,H)$ denote the space of self-adjoint (possibly unbounded) operators on $H$ whose domain is $W$ and \textit{which have a non-zero resolvent set}.  If the Assumption~\ref{A:compact embedding} is satisfied then every  $A\in \calS$ has a compact resolvent and, hence, $A$ has a discrete spectrum. In particular, it is Fredholm.

From this point on we always assume Assumption~\ref{A:compact embedding}.

Let $\calA(\RR,W,H)$ denote the space of norm continuous maps $A: \RR\to \calS(W,H)$ which have limits 
\begin{equation}\label{E:limA(t)}
    A^\pm\ = \ \lim_{t\to \pm\infty}\, A(t)
\end{equation}
and the operators $A^\pm:W\to H$ are bijective. This is an open subspace in the Banach space of all maps satisfying \eqref{E:limA(t)} with the norm 
\begin{equation}\label{E:normA}
	\|A\|_\calA \ := \  \sup_{t\in \RR}\, \|A(t)\|.
\end{equation}

Denote by $\calA^1(\RR,W,H)$ the set of those $A\in \calA$ which are continuously differentiable in the norm topology and satisfy 
\[
	\|A\|_{\calA^1}\ := \ \sup_{t\in \RR}\, \big(\|A(t)\|+\|\dot{A}(t)\|\big)\ <\ \infty. 
\]
The space $\calA^1$ is dense in $\calA$.

The following theorem is proven in \cite[Theorem~4.23]{RobbinSalamon95}:

\begin{theorem}\label{T:RobbinSalamon}
There is a unique map $\mu:\calA(\RR,W,H)\to \ZZ$ satisfying the following axioms:

\begin{enumerate}
\item{\em (Homotopy)} \ \  $\mu$ is constant on the connected components (in norm topology) of $\calA(\RR,W,H)$.
\item{\em (Constant)} \ \  If $A(t)$ is  independent of $t$, then $\mu(A)=0$.
\item{\em (Direct sum)}\ \ $\mu(A_1\oplus A_2)= \mu(A_1)+\mu(A_2)$. 
\item{\em (Normalization)}\ \  For $W=H=\RR$ and $A(t)=\arctan(t)$, we have $\mu(A)=1$. 
\end{enumerate}
\end{theorem}

\begin{remark}\label{R:catenation}
Robbin and Salamon include an extra {\em catenation} axiom in their theorem. But then, \cite[Prop.~4.26]{RobbinSalamon95} they show, that this axiom follows from the others 4 axioms and, hence, can be omitted. Since for our main object of study – the half-spectral flow –  the catenation axiom does not make sense, we prefer to formulate Theorem~\ref{T:RobbinSalamon} without it.  
\end{remark}

\subsection{The spectral flow}\label{SS:spflow}
For $A(t)\in \calA$ let $\spf A$ denote the spectral flow of $A(t)$. There are many equivalent ways to define $\spf A$. Let us recall the definition given in \cite{RobbinSalamon95}. 

 We say that $t\in \RR$ is a {\em crossing} for $A$ if $\ker A(t)\not= \{0\}$.  
First, consider the case when $A\in \calA^1$. If $t$ is a crossing for $A$, we define the {\em crossing operator}
\[
	\Gamma(A,t)\ := \ P(A,t)\dot{A}(t)P(A,t){\big|_{\ker A(t)}},
\]
where $P(A,t):H\to H$ denotes the orthogonal projection onto the kernel of $A(t)$. We say that the crossing $t$ is {\em regular}  if the crossing operator $\Gamma(A,t)$ is non-degenerate. A regular crossing is  {\em simple} if $\dim\ker A(t)= 1$. 

If $t$ is a regular crossing for $A$  we denote by $\sign \Gamma(A,t)$ the signature (number of positive minus the number of negative eigenvalues) of $\Gamma(A,t)$. 

If $t$ is a simple crossing there is a unique real-valued function $\lambda(s)$, defined near $t$, such that $\lambda(s)$ is an eigenvalue of $A(s)$ for all $s$ and $\lambda(t)=0$. We call $\lambda$ the {\em crossing eigenvalue}. Then 
\[
	\sign \Gamma(A,t) \  = \ \sign \dot{\lambda}(t).
\]

If $A\in \calA^1$ has only regular crossings then we define 
\begin{equation}\label{E:sfreguar}
	\spf A\ = \ \sum_t \sign \Gamma(A,t),
\end{equation}
where the sum is taking over all crossings of $A$. 
Thus, if $A(t)$ has only simple crossings 
\begin{equation}\label{E:sfsimple}
	\spf A\ = \ \sum_t \sign \dot{\lambda}(t).
\end{equation}

The following result is a combination of several theorems from \cite{RobbinSalamon95}:

\begin{proposition}\label{P:sfcrossings}
\begin{enumerate}
\item In every connected component  of $\calA(\RR,W,H)$ there  is an $A\in \calA^1$ which has only simple crossings. 

\item   If $A_0,A_1\in \calA^1$ belong to the same connected component of $\calA(\RR,W,H)$ and have only regular crossings then  $\spf A_0= \spf A_1$. 
\end{enumerate}
\end{proposition}

We now define the spectral flow of $A\in \calA(\RR,W,H)$ as the spectral flow of any of the operators $A_1\in \calA^1(\RR,W,H)$ which lies in the same connected component of $\calA(\RR,W,H)$ as $A$ and has only regular crossings.  It follows from Proposition~\ref{P:sfcrossings}(ii) that $\spf A$ is well defined and is constant on connected components of $\calA(\RR,W,H)$. Moreover, it is easy to see that it satisfies the other hypothesis of Theorem~\ref{T:RobbinSalamon}. Thus we obtain the following theorem (cf. \cite[Lemma~4.27]{RobbinSalamon95})

\begin{theorem}\label{T:spflow=mu}
The unique function $\mu(A)$ defined in Theorem~\ref{T:RobbinSalamon} is equal to the spectral flow of $A$
\end{theorem}

 \subsection{The index and the spectral flow}\label{SS:index}
Consider the spaces 
\[
	\mathscr{H}\ :=\  L^2(\RR,H), \qquad \mathscr{W}\ := \ L^2(\RR,W)\cap W^{1,2}(\RR,H)
\]
equipped with the norms 
\[\begin{aligned}
	\|\xi\|_{\mathscr{H}}\ &:= \  \int_{-\infty}^{\infty}\,\|\xi(t)\|_H\,dt,\\
	\|\xi\|_{\mathscr{W}}\ &:= \  \int_{-\infty}^{\infty}\,\big(\,\|\xi(t)\|_W+\|\dot\xi(t)\|_H\,\big)\,dt,
\end{aligned}\]

 Given $A\in \calA^1(\RR,W,H)$ consider the differential operator
 \begin{equation}\label{E:DA}
  	D_A:\xi(t)\ \mapsto \ \frac{d}{dt}\xi(t)\ + \ A(t)\xi(t).
 \end{equation} 
 This is a bounded operator $\mathscr{W}\to \mathscr{H}$. Robbin and Salamon show that it is Fredholm (\cite[Theorem~3.12]{RobbinSalamon95}) and  $\ind D_A$ satisfies the conditions of Theorem~\ref{T:RobbinSalamon} (\cite[page~20]{RobbinSalamon95}). Thus it must be equal to the unique function $\mu$ of Theorem~\ref{T:RobbinSalamon}. Hence, by Theorem~\ref{T:spflow=mu}, we have

 \begin{theorem}\label{T:RobbinSalamon2}
 For $A\in \calA(\RR,W,H)$ we have
 \begin{equation}\label{E:RobbinSalamon2}
 	\ind \Big(\,\frac{d}{dt}+ A(t)\,\Big) \ = \  \spf A.
 \end{equation}
 \end{theorem}

\subsection{A periodic family of operators}\label{SS:periodic}
Consider an element $A\in \calA(\RR,W,H)$, such that there exists a bijective $B\in \calS(W,H)$ such that 
\begin{equation}\label{E:loop}
	A(t)=B \qquad \text{for all} \quad t\not\in[-\pi,\pi].
\end{equation}
Let 
\[
	S^1\ = \ \{ e^{i t}:\,-\pi\le t\le \pi\}
\] 
be the unit circle, and define $\tilde{A}:S^1\to \calS$ by
\[
	\tilde{A}(e^{ i t})\ := \ A(t).	
\]
We denote the space of such maps $\tilde{A}$ by $\Omega\calS(W,H)$. This is a version of a loop space of self-adjoint operators on $H$, namely, we restrict ourselves to loops with certain analytic properties. Then
\[
	\spf(\tilde{A})\ := \ \spf(A)  \ = \ \sum_{-\pi<t<\pi} \sign \Gamma(A,t)
\]
is the spectral flow of $\tilde{A}$ as it is defined in \cite{APS3}. 


It is shown in \cite{APS3} that  the isomorphism $\pi_1(\calS)= \pi_0(\Omega\calS)= \ZZ$ is detected by the spectral flow. In other words, if the spectral flow of a loop $\tilde{A}$ vanishes, then $\tilde{A}$ is homotopic to a constant path in $\Omega\calS$. 

Consider the operator
\[
	D_{\tilde{A}}:  L^2(S^1,W)\cap W^{1,2}(S^1,H) \ \to \ L^2(\RR,H),
\]
defined by 
\[
 	D_{\tilde{A}}:\xi(e^{i t})\ 
 	\mapsto \ \frac{d}{dt}\,\xi(e^{ i t})\ - \ A(e^{i t})\,\xi(e^{it}).
 \]
 This is just the operator \eqref{E:DA} with periodic boundary conditions. Thus, cf. \cite[p.~95]{APS3}
 \[
 	\ind A \ = \ \ind \tilde{A}.
 \]
 Hence, we recover from from \ref{T:RobbinSalamon2} the following result, originally proven in \cite{APS3}:
 \begin{equation}\label{E:APSindex}
 	\ind D_{\tilde{A}} \ = \ -\, \spf \tilde{A}.
 \end{equation}
 Equivalently, 
 \begin{equation}\label{E:APSindex2}
 	\ind \left(\, \frac{d}{dt}+A\,\right) \ = \  \spf \tilde{A}.
 \end{equation}
\section{The half-spectral flow of a family of operators on real line}\label{S:tauabsractflow}

In this section, we consider $\tau$-symmetric paths of self-adjoint operators. Because of the symmetry, the usual spectral flow vanishes. We define a secondary invariant {\em half-spectral flow}, with values in $\ZZ_2$. In   \cite{DollSB21}  this invariant is called the {\em odd half-spectral flow}. This part of the section roughly follows \cite{DeNittisSB15}, but we consider a slightly different setting because of our future applications. 

Theorem~\ref{T:uniqtau} states that several properties of half-spectral flow define it uniquely. This is an analog of Theorem~\ref{T:RobbinSalamon}.  If $A$ is a $\tau$-symmetric path, then the operator $D_A$, cf. \eqref{E:DA} is $\tau$-symmetric (in the sense of Section~\ref{S:tauindex}) with respect to some anti-unitary anti-involution $\tau$. We use the uniqueness of the spectral flow to show that $\ind_\tau D_A$ is equal to the half-spectral flow of $A$. This is an analog of the Robbin–Salamon Theorem~\ref{T:RobbinSalamon2} for our secondary invariants. This result should be compared to \cite[\S8]{BourneCareyLeschRennie22}, though we consider different paths of operators. We discuss the relationship of our result with \cite{BourneCareyLeschRennie22} in Section~\ref{SS:compareBCLR}.

\subsection{An anti-linear anti-involution}\label{SS:involtion}

As in Section~\ref{SS:nonequiv}, we assume that $W\hookrightarrow H$ is a compact embedding of a Hilbert space $W$ into a Hilbert space $H$. We identify $W$ with its image in $H$.
Let $\alpha:H\to H$ be an \textit{anti-linear anti-involution} and assume that $\alpha(W)=W$.

The map $\alpha$ defies anti-linear maps on the spaces of functions $\calW$ and $\calH$ of Section~\ref{SS:index}, which we also denote by $\alpha$.  Define the maps $\tau:\calW\to \calW$ and $\tau:\calH\to \calH$ by 
\begin{equation}\label{E:tau}
	\tau:\, \xi(t)\ \mapsto \ \alpha \xi(-t).
\end{equation}
By a slight abuse of the language we denote by the same letter the conjugation
$\tau:\calA(\RR,W,H)\to \calA(\RR, W,H)$, 
\begin{equation}\label{E:tau2}
		\tau:\, A(t)\ \mapsto \ \tau\, A(t)\, \tau^{-1} \ = \  \alpha\, A(-t)\, \alpha^{-1}.
\end{equation}

Let 
\begin{equation}\label{E:Atau}
	\At(\RR,W,H)\ := \ \big\{\,A\in \calA(\RR,W,H):\, \tau(A)=A\,\big\}.
\end{equation}
denote the set of $\tau$-invariants in $\calA$ and set set $\At^1:= \At\cap \calA^1$.   Recall that if $A\in \calA$ then $A(t)$ has a discrete spectrum and compact resolvent for all $t\in \RR$.

\subsection{The half-spectral flow of a $\tau$-invariant path of operators}\label{SS:tauflow}
Suppose $A\in \At$. If $t\in \RR$ is a crossing for $A$, then $-t$ is also a crossing for $A$ and 
\begin{equation}\label{E:tauGamma}
	\Gamma(A,t) \ = \  -\,\alpha\, \Gamma(A,-t) \, \alpha^{-1}.
\end{equation}
Hence, $\sign\Gamma(A,t)= -\sign\Gamma(A,-t)$ and, if 0 is a crossing,  $\sign\Gamma(A,0)=0$. We conclude that {\bf the spectral flow of any $A\in \At$ is equal to zero}. In this section we define  a {\em secondary invariant} of $A$ –– the $\ZZ_2$-valued half-spectral flow.

Before giving the definition and discussing the properties of the half-spectral flow, let us put it in a perspective of the works of Atiayh–Singer \cite{AtSinger69} and Carey-Phillips-Schulz-Baldes \cite{CareyPhillipsSB19}. 

Recall that the loop space $\Omega\calS(W,H)$ was introduced in Section~\ref{SS:periodic}. Define $\tau:\Omega\calS\to \Omega\calS$ by 
\[
	\tau\big(\tilde{A}(e^{it})\big) \ : = \ \alpha\, \tilde{A}\big(e^{-it}\big)\,\alpha, 
\]
and let
\[
	\Omega_\tau\calS(W,H) \ := \ 
	\big\{\,\tilde{A}\in \Omega\calS(W,H):\, \tau( \tilde{A})=\tilde{A}\,\big\},
\]
be the space of  $\tau$-invariant loops. 

Since the spectral flow of $\tilde{A}$ vanishes, it follows from \cite{AtSinger69} that is $\tilde{A}$ is homotopic in the space $\Omega\calS$ to a constant path. However, if the half-spectral flow does not vanish, this homotopy cannot be chosen inside the space $\Omega_\tau\calS$.

\subsection{The uniqueness of half-spectral flow}\label{SS:uniqtau}
The first main result of this section is the following analog of Theorem~\ref{T:uniqtau}:

\begin{theorem}\label{T:uniqtau}
There is a unique map $\mu_\tau:\At(\RR,W,H)\to \ZZ_2$ satisfying the following axioms:

\begin{enumerate}
\item{\em (Homotopy)} \ \  $\mu_\tau$ is constant on the connected components (in norm topology) of \linebreak $\At(\RR,W,H)$.
\item{\em (Constant)} \ \  If $A(t)$ is  independent of $t$, then $\mu_\tau(A)=0$.
\item{\em (Direct sum)}\ \  $\mu_{\tau_1\oplus\tau_2}(A_1\oplus A_2)= \mu_{\tau_1}(A_1)+\mu_{\tau_2}(A_2)$. 
\item{\em (Normalization)}\ \  For $W=H=\CC^2$, 
\[
	\alpha: \CC^2\ \to \ \CC^2, \qquad \alpha:\,(z_1,z_2)\ \mapsto \ (\bar{z}_2,-\bar{z}_1).
\]
and 
\begin{equation}\label{E:taunormalization}
	A(t)\ = \ 
	\begin{pmatrix}
		\arctan(t)&0\\0&-\arctan(t)
	\end{pmatrix},
\end{equation}
we have $\mu_\tau(A)=1$. 
\end{enumerate}
The number $\mu_\tau(A)$ is called the {\bf half-spectral flow} (or the {\em $\tau$-spectral flow}) of A. 
\end{theorem}

The proof of the theorem is postponed to Section~\ref{SS:pruniqtau}.

\subsection{The half-spectral flow}\label{SS:tauspflow}
We present an explicit construction of the half-spectral flow in terms of the crossing operators. Recall that $t\in\RR$ is called a crossing for $A(t)$ if $\ker A(t)\not=\{0\}$. We say that a crossing $t$ is \textit{isolated} if the exists $\delta>0$ such that $t$ is the only crossing for $A(t)$ on the interval $(t-\delta,t+\delta)$. Note that if $A\in A^1_\tau$ and $t$ is a regular crossing (cf. Section~\ref{SS:spflow}) then $t$ is an isolated crossing.   

First, assume that $A\in \calA_\tau$ has only isolated crossings. 
Since we are only interested in mod 2 value of the flow, we can replace the signature  of the crossing operator with the rank of the spectral projection $P(A,t)$:
\begin{equation}\label{E:Gamma=P}
		\sign \Gamma(A,t)\ = \ \rank P(A,t)\qquad \mod 2.
\end{equation}
The space $\ker P(A,0)$ is $\alpha$-invariant. It follows from  Lemma~\ref{L:even}  that $\rank P(A,0)$ is even. As opposed to the crossing operators $\Gamma(A,t)$, the projections $P(A,t)$ are defined for all $A\in \At$ (and not only for those in $\At^1$). If $A\in \At$  has only isolated crossings we define 
\begin{equation}\label{E:tauspflow}
	\spf_\tau (A) \ := \  \ \sum_{t< 0} \rank P(A,t) \ + \ \frac12\rank P(A,0), \qquad \mod 2. 
\end{equation}
We have the following analogue of Proposition~\ref{P:sfcrossings}: 
\begin{proposition}\label{P:sfcrossings tau}
If $A_0, A_1\in \calA_\tau(\RR,W,H)$ belong to the same connected component and have only isolated crossings then  $\spf_\tau A_0= \spf_\tau A_1$. 
\end{proposition}
The proof is given in Section~\ref{SS:prsfcrossings tau}.  We now define the half-spectral flow of $A\in \calA_\tau(\RR,W,H)$ as the spectral flow of any of the operators $A_1$ with isolated crossings lying in  the same connected components of $\calA_\tau$ as $A$. By Proposition~\ref{P:sfcrossings tau} this definition is independent of the choice of $A_1$.
\begin{theorem}\label{T:tauspflowhommotopy}
The map $\spf_\tau:\At(\RR,W,H)\to \ZZ_2$ satisfies the conditions of Theorem~\ref{T:uniqtau} and, hence, is the half-spectral flow. 
\end{theorem}
The proof of the theorem is given in Section~\ref{SS:prtauspflowhommotopy}.

\subsection{The $\tau$–index of $D_A$}\label{SS:tauindex}
Let $A\in \At$ and let $D_A$ be as in \eqref{E:DA}. Then 
\begin{equation}\label{E:tauDAtau}
    	\tau\, D_A\,\tau^{-1} \ = \ D_A^*,
\end{equation}
where $\tau$ is defined in \eqref{E:tau}. Hence $D_A$ is $\tau$-invariant and its $\tau$-index is defined by \eqref{E:tauindex}, 
\[
    \ind_\tau D_A \ = \ \dim \ker D_A \qquad \mod \ 2.
\]

We have the following analog of the Robbin-Salamon Theorem~\ref{T:RobbinSalamon2}:

\begin{theorem}\label{T:tauRS}
The map $\ind_\tau: \At(\RR,W,H)\to \ZZ_2$ satisfies the conditions of Theorem~\ref{T:uniqtau} and, hence, is equal to the half-spectral flow:
 \begin{equation}\label{E:tauindex=tausf}
 	\ind_\tau \Big(\,\frac{d}{dt}+ A(t)\,\Big) \ = \  \spf_\tau A  \qquad \mod 2.
 \end{equation}
\end{theorem}

The proof of Theorem~\ref{T:tauRS} is presented in Section~\ref{SS:prtauRS}.

\subsection{Comparison with the spectral flow of a path of skew-adjoint operators}\label{SS:compareBCLR}
For $A(t)\in \At(\RR,W,H)$ suppose for simplicity that the operator $A(0)$ is invertible.  Set 
\begin{equation}\notag
    A_{<0}(t):\ = \ \begin{cases}
        A(t)\quad&\text{for}\ \ t\le0;\\
        A(0)\quad&\text{for}\ \ t>0,
    \end{cases}
    \qquad\qquad 
    A_{>0}(t):\ = \ \begin{cases}
        A(t)\quad&\text{for}\ \ t\ge0;\\
        A(0)\quad&\text{for}\ \ t<0.
    \end{cases}
\end{equation}
Consider the ``doubling" of the path $A$:
\begin{equation}\notag
    \hat{A}(t)\ := \ \begin{pmatrix}
        0&A_{<0}(t)\\
        -\alpha\,A_{>0}(-t)\,\alpha^{-1}&0
    \end{pmatrix}.
\end{equation}
This is a path of skew adjoint operators. 
The $\ZZ_2$-valued spectral flow of $\hat{A}$ considered in \cite{BourneCareyLeschRennie22} is exactly equal to $\spf_\tau A$. Theorem~8.1 of  \cite{BourneCareyLeschRennie22} expresses this spectral flow as a $\ZZ_2$-valued index of a certain ``suspension" operator of the form similar to $\frac{d}{dt}+A$. Using the gluing formula for the $\tau$-index, \cite{BrSaeediYan25}, one can show that the latter index is equal to $\ind_\tau \big(\frac{d}{dt}+ A(t)\big)$. Thus the half-spectral flow of $A$ is equal to the $\tau$-index of $\frac{d}{dt}+A$. Combining it with Theorem 8.1 of \cite{BourneCareyLeschRennie22}, we see that the $\ZZ_2$-valued indexes of $\frac{d}{dt}+A$ and of the operator in \cite{BourneCareyLeschRennie22} coincide. This can also be shown by a direct, but non-trivial argument. 

In conclusion, our Theorem~\ref{T:tauRS} is closely related to Theorem 8.1 of \cite{BourneCareyLeschRennie22}, but the relationship is not very straightforward.  This is why we preferred to give a direct proof of our result in the next section, by adopting the original argument of Robbin-Salamon \cite{RobbinSalamon95} to our $\ZZ_2$-valued invariants.

\section{Proof of $\ZZ_2$-valued version of the Robbin-Salamon theorem}\label{S:proofs of three theorems}

This section is dedicated to the proofs of Proposition~\ref{P:sfcrossings tau} and Theorems~\ref{T:tauspflowhommotopy}, \ref{T:tauRS}, and \ref{T:uniqtau}.

\subsection{Proof of Proposition~\ref{P:sfcrossings tau}}\label{SS:prsfcrossings tau}

Suppose $A_0$ and $A_1$ lie in the same connected component of $\calA_\tau(\RR,W,H)$ and let $A_s$ ($s\in[0,1]$) be a continuous path in $\calA_\tau$  connecting between them such that all $A_s$ have only isolated crossings.  We need to prove that $\spf_\tau(A_s)$ is independent of $s$. If $t=0$ is not a crossing point for some $A_{s_0}$, then there exists $\epsilon>0$ such that $0$ is not a crossing for all $A_s$, $s\in (s_0-\epsilon,s_0+\epsilon)$. Then the path
\[
	\tilde{A}_s(t) \ := \ 
	\begin{cases}
	 A_s(t) \quad&\text{for}\quad t\le 0,\\
	 A_s(0)\quad&\text{for}\quad t>0,
	\end{cases}
\]
is in $\calA(\RR,W,H)$ for $s\in  (s_0-\epsilon,s_0+\epsilon)$. The $\tau$-spectral flow of $A_s$ is equal modulo 2 to the usual $\ZZ$-valued spectral flow  $\spf(\tilde{A}_s)$. Hence, it is independent of $s\in   (s_0-\epsilon,s_0+\epsilon)$ by the stability of the usual spectral flow under homotopy. 

If $0$ is a crossing for $A_{s_0}$, then there exist $\epsilon, \delta>0$ such that
\begin{enumerate}
    \item there are no non-zero eigenvalues of $A_{s_0}(0)$ in the interval $(-\delta,\delta)$;
    \item there are no crossings $t$ of $A_{s_0}$ with $0<|t|\le\epsilon$;
    \item $\pm\delta$ are not in the spectrum of $A_s(t)$ for all $s\in (s_0-\epsilon,s_0+\epsilon)$, $t\in (-\epsilon,\epsilon)$.
\end{enumerate}
Then the paths 
\[
	\tilde{A}_s(t) \ := \ 
	\begin{cases}
	 A_s(t) \quad&\text{for}\quad t\le -\epsilon,\\
	 A_s(-\epsilon)\quad&\text{for}\quad t>-\epsilon,
	\end{cases}
\]
and 
\[
	\hat{A}_s(t) \ := \ 
	\begin{cases}
	 A_s(-\epsilon) \quad&\text{for}\quad t<-\epsilon,\\
  A_s(t) \quad&\text{for}\quad -\epsilon\le t\le \epsilon,\\
	 A_s(\epsilon)\quad&\text{for}\quad t>\epsilon,
	\end{cases}
\]
are in $\calA(\RR,W,H)$ for all $s\in (s_0-\epsilon,s_0+\epsilon)$, $t\in (-\epsilon,\epsilon)$. By \eqref{E:tauspflow} , we have 
\begin{equation}\label{E:tauspflowtildA}
	\spf_\tau (A_s) \ = \  \spf(\tilde{A}_s) \ + \ \sum_{-\epsilon<t< 0} \rank P(A_s,t) \ + 
 \frac12\rank P(A_s,0), \qquad \mod 2.
\end{equation}
Since $t=0$ is not a crossing for  $\tilde{A}_s$ ($s\in (s_0-\epsilon,s_0+\epsilon)$, $\spf(\tilde{A}_s)$ is independent of $s$ on this interval.

From  \eqref{E:sfreguar} and \eqref{E:Gamma=P}, we see that the mod 2 reduction of  usual $\ZZ$-valued spectral of $\hat{A}_s$ is given by
\[
    \spf (\hat{A}_s) \ = \  \ \sum_{-\epsilon<t< 0} \rank P(A_s,t) 
 \ + \ 
 \rank P(A_s,0) \ + \sum_{0<t< \epsilon} \rank P(A_s,t) \qquad\text{mod}\quad 2.
\]
By \eqref{E:tau2}, $P(A_s,t)= \alpha\, P(A_s,-t)\,\alpha^{-1}$. It follows that
\[
\sum_{-\epsilon<t< 0} \rank P(A_s,t) 
 \ = \ \sum_{0<t< \epsilon} \rank P(A_s,t)
\]
and $\spf(\hat{A}_{s})=  \rank P(A_{s},0)$ mod 2. From the stability of the usual spectral flow under homotopy, we have $\spf(\hat{A}_{s_0})= \spf(\hat{A}_s)$ and, hence, 
\begin{equation}\label{E:spflowhatA} 
 \rank P(A_s,0)\ = \  \rank P(A_{s_0},0), \qquad \text{mod}\quad 2.
\end{equation}

Substituting this equality into \eqref{E:tauspflowtildA} and using that $\spf(\tilde{A}_s)$ is independent of $s\in (s_0-\epsilon,s_0+\epsilon)$, we conclude that 
\[
    \spf_\tau (A_s) \ = \  \spf(\tilde{A}_{s_0}) \  +\  
 \frac12\rank P(A_{s_0},0) \ = \ \spf_\tau(A_{s_0}), \qquad \mod 2.
\]
\hfill$\square$

We now pass to the proof of Theorems~\ref{T:tauspflowhommotopy} and \ref{T:tauRS}.
\subsection{The normalization axiom}\label{SS:findim}
Let $W=H= \CC^2$ and let $\alpha$ and $A$ be as in Theorem~\ref{T:uniqtau}(iv). Then the eigenvalues of $A(t)$ are $\pm\arctan(t)$. By \eqref{E:tauspflow}, $\spf_\tau(A)=1$. Hence, we obtain the following

\begin{lemma}\label{L:spfnormalization}
$\spf_\tau$ satisfies the normalization condition of Theorem~\ref{T:uniqtau}.
\end{lemma}

Also, by Lemma~2.7 of \cite{BrSaeedi24index} we have

\begin{lemma}\label{L:indexnormalizationtau}
The map $A\mapsto \ind_\tau D_A$ satisfies the normalization condition of Theorem~\ref{T:uniqtau}.
\end{lemma}

\subsection{A constant path}\label{SS:constantpath}
Let us consider the case when $A(t)=A_0\in \At$ is independent of $t$. Then $A_0$ is bijective.  It follows that $\spf_\tau(A)=0$, i.e. $\spf_\tau$ satisfy axiom (ii) of Theorem~\ref{T:uniqtau}. 

Further, if $D_A\xi(t)= 0$, then $\xi(t)= e^{-A_0t}\xi(0)$, which is not square-integrable, unless $\xi(0)=0$. Hence, $\ker{}D_A= \{0\}$. The following lemma summarizes this subsection:

\begin{lemma}\label{L:constanttau}
The maps $A\mapsto \spf_\tau(A)$ and $A\to \ind_\tau D_A$ satisfy axiom (ii) of Theorem~\ref{T:uniqtau}.
\end{lemma}

\subsection{Proof of Theorem~\ref{T:tauspflowhommotopy}}\label{SS:prtauspflowhommotopy}
By Proposition~\ref{P:sfcrossings tau}, $\spf_\tau$ satisfies the homotopy axiom (i) of Theorem~\ref{T:uniqtau}. By Lemmas~\ref{L:spfnormalization} and \ref{L:constanttau}, $\spf_\tau$ satisfies axioms (ii) and (iv) of Theorem~\ref{T:uniqtau}.  It is clear, that it also satisfies the direct sum axiom (iii). 
\hfill$\square$
\subsection{Proof of Theorem~\ref{T:tauRS}}\label{SS:prtauRS}
By Lemmas~\ref{L:indexnormalizationtau} and \ref{L:constanttau}, $\ind_\tau D_A$ satisfies axioms (ii) and (iv) of Theorem~\ref{T:uniqtau}.  It is clear, that it also satisfies the direct sum axiom (iii). By Theorem~\ref{T:homotopyindex} it also satisfies the homotopy axiom. Thus it coincides with the half-spectral flow by Theorem~\ref{T:uniqtau}.
\hfill$\square$
    
\subsection{The proof of Theorem~\ref{T:uniqtau}}\label{SS:pruniqtau}
Theorem~\ref{T:tauspflowhommotopy} states that $\spf_\tau$ satisfies the axioms of Theorem~\ref{T:uniqtau}. Hence, the existence of $\mu_2$ is proven. 

To prove the uniqueness we need to show that any $\mu_2$, satisfying the axioms of Theorem~\ref{T:uniqtau}, is equal to $\spf_\tau$.

A verbatim repetition of the  Robbin-Salamon proof of Theorem~\ref{T:RobbinSalamon}, cf. \cite[\S4]{RobbinSalamon95}, shows that there exists a unique map $\mu_2:\calA(\RR,W,H)\to \ZZ_2$, satisfying the axioms (i)-(iv) of Theorem~\ref{T:RobbinSalamon}. This $\mu_2$ is just the mod 2 reduction of $\mu$. 

For any $B\in\calA$  the {\em symmetrization} of $B$ is the path 
\[
    \tilde{B}(t)\ := \ 
    \begin{cases}
        B(t)\oplus \alpha\, B(-t)\, \alpha^{-1} \qquad&\text{for}\quad t\le0\\
        \alpha\, B(-t)\, \alpha^{-1}\oplus B \qquad&\text{for}\quad t>0
    \end{cases}
    \ \  \in \ \ \At(\RR,W\oplus W,H\oplus H).
\]
We claim that the map $B\mapsto \mu_\tau(\tilde{B})$ satisfies all the axioms for $\mu_2$. Indeed, if $B_1$ is homotopic to $B_2$ in $\calA$, then $\tilde{B}_1$ is homotopic to $\tilde{B}_2$ in $\At$, which implies the homotopy axiom. The rest of the axioms follow immediately from the corresponding axioms for $\mu_\tau$. It follows now from the uniqueness  of $\mu_2$ that 
\begin{equation}\label{E:mu2=mutau}
    \mu_2(B)\ = \ \mu_\tau(\tilde{B}).
\end{equation}

We want to compute $\mu_\tau(A)$  in terms of $\mu_2$ for any $A\in \At$. First, consider a path $A\in \At$, which is bijective at 0 and set
\[
	A_{\le0}(t)\ := \ 
	\begin{cases}
	A(t)\quad&\text{for}\quad t\le 0,\\
	A(0)\quad&\text{for}\quad t> 0.
	\end{cases}
\]
Then $A_{\le0}\in \calA(\RR,W,H)$ and the symmetrization $\tilde{A}_{\le0}$ is the direct sum of $A$ and the constant path $A(0)$:
\[
    \tilde{A}_{\le0} \ = \ A\oplus A(0).
\]
Thus, from the direct sum and the constant axiom for $\mu_\tau$ we conclude that $\mu_\tau(A)= \mu_\tau(\tilde{A}_{\le0})$. Using \eqref{E:mu2=mutau}, we get
\begin{equation}\label{E:mutau=mu2}
	\mu_\tau(A)\ = \ \mu_\tau(\tilde{A}_{\le 0}) \ = \ \mu_2(A_{\le0})\ \equiv \ \sum_{t< 0} \rank P(A,t)
	 \ = \ \spf_\tau(A).
\end{equation}
Here the third equality holds modulo 2 and follows from Theorem~\ref{T:spflow=mu} and \eqref{E:Gamma=P}.

Consider a general $A\in \At$. Then, by Lemma~\ref{L:even}, the kernel of $A(0)$ is even dimensional, and there is a subspace $L$ in it such that $\ker A(0)\ = \ L\oplus\alpha{L}$. Then there exists $\epsilon>0$ and  a path $A'\in \At$ homotopic to $A$, such that $A'(t)$ is bijective for all $t\in (-\epsilon,\epsilon)$, $\ker{}A(-\epsilon)= L$, and $\ker{}A(\epsilon)= \tau(L)$.  Then 
\[
	\rank P(A',-\epsilon)\ = \  \dim L \ = \  \frac12\,\rank P(A,0). 
\]
It now follows from \eqref{E:mutau=mu2}, that $\mu_\tau(A)$ is given by \eqref{E:tauspflow}.
\hfill$\square$

\section{The $\ZZ_2$-valued index and the $\ZZ_2$-valued spectral flow on a circle}\label{S:APScircle}

In this section, we consider a $\tau$-invariant family of self-adjoint operators  
 $A(t)$, parametrized by $t\in S^1$ and prove an analog of Theorem~\ref{T:tauRS}
for it. For the usual $\ZZ$-valued spectral flow this result basically follows from the APS index theorem, cf. \cite[\S7]{APS3}. This theorem is unavailable for the $\tau$-index. Instead, we deduce the result for the periodic families from Theorem~\ref{T:tauRS}.

Throughout this section, the Hilbert spaces $W$ and $H$ are as in Section~\ref{SS:nonequiv}.

\subsection{Families of operators on an interval and on a circle}\label{SS:intervalfamily}
Let $\calA([-\pi,\pi],W,H)$ be the space of norm continuous maps $A:[-\pi,\pi]\to \calS(W,H)$. For $A\in \calA([-\pi,\pi],W,H)$ set
\begin{equation}\label{E:tildeAextension}
    \tilde{A}(t)\ := \ \begin{cases}
        A(-\pi),\qquad &\text{for}\quad t<-\pi,\\
        A(t),\qquad &\text{for}\quad -\pi\ge t\le\pi,\\
        A(\pi),\qquad &\text{for}\quad t>\pi.
    \end{cases}
\end{equation}
Then $\tilde{A}\in \calA(\RR,W,H)$ iff $A(\pi)$ and $A(-\pi)$ are bijective. 

Let 
\[
    \calA(S^1,W,H)\ := \ \big\{\, A\in \calA([-\pi,\pi],W,H):\, A(-\pi)=A(\pi)\,\big\}.
\]
We view elements of $\calA(S^1,W,H)$ as families of operators on a circle. 

\sloppy Suppose that $\alpha:H\to H$ is an anti-unitary anti-involution. As in Section~\ref{SS:involtion}, we define anti-linear anti-involution $\tau$ on the spaces $\calA([-\pi,\pi],W,H)$ and $\calA(S^1,W,H)$. Let    $\At([-\pi,\pi],W,H)$ and $\At(S^1,W,H)$ denote the subspaces of $\tau$-invariant elements in \linebreak[4]
$\calA([-\pi,\pi],W,H)$ and $\calA(S^1,W,H)$ respectively.  

Note that if $A\in \calA_\tau(S^1,W,H)$ then 
\begin{equation}\label{E:A(pi) odd symmetric}
    A(-\pi)\ = \ A(\pi)\ = \ \alpha\,A(\pi)\, \alpha^{-1},
\end{equation}
i.e. the operator $A(\pi)= A(-\pi)$ is odd symmetric.

\subsection{The half-spectral flow on the circle}\label{SS:halfspf circle}
For $A\in \calA_\tau(S^1,W,H)$  which have only isolated crossings, we define the half-spectral flow by 
\begin{equation}\label{E:tauspflowcircle}
	\spf_\tau (A) \ := \  \ \sum_{-\pi<t< 0} \rank P(A,t) \ + \ \frac12\rank P(A,0)\ + \ \frac12\rank P(A,-\pi), \qquad \mod 2. 
\end{equation}
A verbatim repetition of the arguments in the proof of Proposition~\ref{P:sfcrossings tau} shows that if $A_0$ and $A_1$ lie in the same connected component of  $\calA_\tau(S^1,W,H)$ and have isolated crossings, then  $\spf_\tau(A_0)= \spf_\tau(A_1)$. As in Section~\ref{SS:tauspflow}, for a general $A\in\calA_\tau(S^1,W,H)$  we define the half-spectral flow $\spf_\tau(A)$ as the half-spectral flow of any $A_1$ which has isolated crossing and lies in the same connected component of  $\calA_\tau(S^1,W,H)$ as $A$.

\subsection{The operator $D_A$ on an interval and on a circle}\label{SS:DAperiodic}
Let $\calH$ and $\calW$ be as in Section~\ref{SS:index} and let $\calHp$ and $\calWp$ be similar spaces with $\RR$ replaced by $[-\pi,\pi]$ (thus $\calHp:= L^2([-\pi,\pi],H)$, etc.). Let 
\[
    \calW_{S^1}\ : = \ \big\{\xi\in \calWp:\, \xi(-\pi)=\xi(\pi)\,\big\}. 
\]
We view $\calW_{S^1}$  as a space of $H$-valued functions on  $S^1$. We also view $\calH_{[-\pi,\pi]}$ as the space of square-integrable $H$-valued functions on $S^1$ and use the notations $\calH_{[-\pi,\pi]}$ and $\calH_{S^1}$ for this space interchangeably.  

Suppose $A\in \calA_\tau(S^1,W,H)$ and consider the operator 
\begin{equation}\label{E:DAcircle}
  	D_{A,S^1}:\xi(t)\ \mapsto \ \frac{d}{dt}\xi(t)\ + \ A(t)\xi(t), \qquad \xi\in \calW_{S^1}.
 \end{equation} 
It is a bounded odd symmetric operator $\calW_{S^1}\to \calH_{S^1}$. If we view $\calW_{S^1}$ as a subspace of $\calW_{[-\pi,\pi]}$ then $D_A$ becomes an operator on the interval $[-\pi,\pi]$ with periodic boundary conditions. It is well known that this operator is Fredholm. When $H$ is a space of square-integrable sections of a bundle over a compact manifold this follows from the standard description of elliptic boundary conditions for PDE, cf., for example, \cite[\S7.3]{BarBallmann12} (in \cite{BarBallmann12} the periodic boundary conditions are called ``transmission conditions". cf. \cite[Example~7.28]{BarBallmann12}). The abstract description of elliptic boundary conditions is very similar, but simpler, cf. \cite[Appendix~A]{BandaraGoffengRozen25}).

It follows from Theorem~\ref{T:homotopyindex} that the map $A\to \ind_\tau(D_A)$ is constant on the connected components of $\calA_\tau(S^1,W,H)$.

The main result of this section is the following analog of Theorem~\ref{T:RobbinSalamon2}

\begin{theorem}\label{T:RScircle}
Suppose $A\in \calA_\tau(S^1,W,H)$. The $\tau$-index of the operator $D_{A,S^1}:\calW_{S^1}\to \calH_{S^1}$ is equal to the half-spectral flow $\spf_\tau(A)$.
\end{theorem}

The rest of this section is occupied with the proof of this theorem.

\subsection{The case when $A(-\pi)=A(\pi)$ is invertible}\label{SS:circle_invertible_case}
Suppose, first, that $A(-\pi)= A(\pi)$ is invertible. Then the operator \eqref{E:tildeAextension} is in $\calA_\tau(\RR,W,H)$ and $\spf_\tau(\tilde{A})= \spf_\tau(A)$. Hence, from Theorem~\ref{T:RobbinSalamon2} we conclude that 
\begin{equation}\label{E:indDtildeA=sfA}
    \ind_\tau D_{\tilde{A}}\ = \ \spf_\tau A.
\end{equation}

We now recall the abstract Atiyah-Patodi-Singer (APS) boundary conditions, cf. \cite[\S4]{DungenRonge21}. Let $\calH_{<0}\subset H$ (resp. $\calH_{>0}\subset H$)  denote the span of eigenvectors of $A(-\pi)=A(\pi)$ with negative (resp. positive) eigenvalues. Set
\begin{equation}
  \begin{aligned}
          \calW_{APS}\ &:= \ \big\{\,f\in \calW_{[-\pi,\pi]}:\, f(-\pi)\in \calH_{<0},
    \ f(\pi)\in \calH_{>0}\,\big\},\\
    \calW_{APS}^\dagger\ &:= \ \big\{\,f\in \calW_{[-\pi,\pi]}:\, f(-\pi)\in \calH_{>0},
    \ f(\pi)\in \calH_{<0}\,\big\}.
  \end{aligned}
\end{equation}
Then $\tau: \calW_{APS}\to \calW_{APS}^\dagger$.

The operator $D_{A,APS}:\calW_{APS}\to \calH$ is called the operator $D_A$ with the {\em APS boundary conditions}. The adjoint of this operator is the operator
\[
    D_{A,APS}^*= -\frac{d}{dt}+A(t):\calW_{APS}^\dagger\to \calH.
\]
It follows that $\tau D_{A,APS}\ \tau^{-1}= D_{A,APS}^*$, i.e. {\em the operator $D_{A,APS}$ is odd symmetric}.

Recall that the operator $\tilde{A}$ is defined in \eqref{E:tildeAextension}.  By  Proposition~4.5 of \cite{DungenRonge21}, the dimension of the kernel of $D_{A,APS}$ is equal to the dimension of the kernel of the operator $D_{\tilde{A}}$.  Hence, the $\tau$-indices of these operators are equal
\begin{equation*}\label{E:DAPS=tildeD}
    \ind_\tau D_{A,APS}\ = \ \ind_\tau D_{\tilde{A}}.
\end{equation*}
Using \eqref{E:indDtildeA=sfA} we obtain
\begin{equation}\label{E:indDAPS=spfA}
        \ind_\tau D_{A,APS}\ = \ \spf_\tau A.
\end{equation}
Theorem~\ref{T:RScircle} for the case when $A(\pi)$ is invertible follows now from the following 

\begin{proposition}\label{P:indAPS=indperiodic}
If $A(\pi)= A(-\pi)$ is invertible, then $\ind_\tau D_{A,APS}= \ind_\tau D_{A,S^1}$.
\end{proposition}
\begin{proof}
In the proof of Theorem 8.17 of  \cite{BarBallmann12} (see also the proof of Theorem~5.11 in \cite{BrShi21odd}) the similar statement is proven for the usual index in the case when $H$ is the space of square-integrable sections over a manifold so that $D_A$ is a PDE. The proof goes by contracting a continuous family of elliptic boundary conditions $B_s$ which interpolates between the periodic and the APS boundary conditions: $B_0$ is the APS boundary conditions and $B_1$ is the periodic boundary conditions. This construction also works in our case, where the proof of the fact that the intermediate boundary conditions $B_s$ are Fredholm is exactly the same as in the PDE case. The details can be found in cf. \cite[Appendix~A]{BandaraGoffengRozen25}). Then one obtains a family of operators  $D_{A,B_s}$ with different domains. One then shows that there is a unitary equivalent continuous family of operators $D_{A,s}$ of operators with domain $\calW_{S^1}$. By construction, all the operators  $D_{A,B_s}$ are odd symmetric. Thus, the proposition follows from the stability of the $\tau$-index (Theorem~\ref{T:homotopyindex}).
\end{proof}

\subsection{The case when $A(\pi)$ is not invertible}\label{SS:circle_general_case}
Suppose  $A(-\pi)= A(\pi)$ is not invertible.  Since $A(\pi)$ has a discrete spectrum,  for small enough $\epsilon>0$,  the operator $A(\pi)+\epsilon$ is invertible. The path $A+\epsilon\in \calA_\tau(S^1,W,H)$ is in the same connected component of $\calA_\tau(S^1,W,H)$ as $A$. Hence 
\begin{equation}\label{E:A=A+epsilon}
    \spf_\tau(A+\epsilon)\ =\ \spf_\tau(A), \qquad
    \ind_\tau(D_A)\ =\  \ind_\tau(D_{A+\epsilon}).
\end{equation}
It is shown in Section~\ref{SS:circle_invertible_case} that $\spf_\tau(A+\epsilon)= \ind(D_{A+\epsilon})$. Theorem~\ref{T:RScircle} follows now from \eqref{E:A=A+epsilon}. \hfill$\square$

\section{The half-spectral flow of a family of Toeplitz operators}\label{S:sfToeplitz}

This section studies a family of odd symmetric Toeplitz operators on a complete Riemannian manifold.  Our main result is that the graded half-spectral flow of such a family is equal to the  $\tau$-index of a certain Callias-type operator. This is a $\ZZ_2$-valued analog of the main result of \cite{Br19Toeplitz}.  In the next section, we apply it to obtain a new proof and a generalization of the bulk-edge correspondence in the Graf-Porta model of topological insulators with time-reversal symmetry \cite{GrafPorta13}.

\subsection{A (generalized) Dirac operator}\label{SS:Dirac}
Let $M$ be a complete Riemannian manifold. Recall from \cite[\S{}II.5]{LawMic89} that a graded bundle $E=E^+\oplus E^-$ over $M$ is called a Dirac bundle if it is a Hermitian vector bundle endowed with the Clifford action 
\[	
	c:T^*M\to \End(E), \qquad 
	\big(\,
	c(v):E^\pm\to E^\mp, \quad c(v)^2=-|v|^2, \quad c(v)^*=-c(v)
	\,\big),
\]
and a Hermitian connection $\n^E=\n^{E^+}\oplus\n^{E^-}$ compatible with the Clifford action. We assume that the Hermitian metric is $\theta^E$-invariant

We extend the Clifford action to the product $E\otimes\CC^k$ and denote by  $D$ the associated (generalized) Dirac operator. In local coordinates, it can be written as $D=\sum_jc(dx^j)\n^E_{\p_j}$. We view $D$ as an unbounded  self-adjoint operator $D:L^2(M,E\otimes\CC^k)\to L^2(M,E\otimes\CC^k)$.

From this point on we make the following

\begin{assumption}\label{A:1}
Zero is an isolated point of the spectrum of $D$. 
\end{assumption}

Let $\HH=\HH^+\oplus \HH^-\subset L^2(M,E\otimes\CC^k)$ denote the kernel of $D$ and let $P:L^2(M,E\otimes\CC^k)\to \HH$ be the orthogonal projection. Here $\HH^\pm$ is the intersection of $\HH$ with $L^2(M,E^\pm\otimes\CC^k)$. We denote by $P^\pm:L^2(M,E^\pm\otimes\CC^k)\to \HH^\pm$ the restriction of $P$ to $L^2(M,E^\pm\otimes\CC^k)$.

\subsection{Matrix valued functions}\label{SS:matrix valued}
Let $BC(M,k)$ denote the Banach algebra of continuous bounded functions $f(x)$ on $M$ with values in the space  $\Herm(k)$ of Hermitian complex-valued $k\times{}k$-matrices.

Let $C_{0}(M,k)$ denote the closure of the subalgebra of functions with compact support in $BC(M,k)$. 

Let $C_g^\infty(M,k)\subset BC(M,k)$ denote the space of smooth functions with values in $\Herm(k)$  which are bounded and such that $df$ vanishes at infinity.  Then $C_0(M,k)\subset C_g(M,k)$. 

For  $f\in C^\infty_g(M;k)$ we denote by $M_f:L^2(M,E\otimes\CC^k)\to L^2(M,E\otimes\CC^k)$ the multiplication by $1\otimes{f}$ and by $M_f^\pm$ the restriction of $M_f$ to $L^2(M,E^\pm\otimes\CC^k)$.

\subsection{The Toeplitz operator}\label{SS:Toeplitz}
Recall that $\HH=\HH^+\oplus\HH^-$ denotes the kernel of $D$ and $P:L^2(M,E\otimes\CC^k)\to \HH$ is the orthogonal projection.

\begin{definition}\label{D:Toeplitz}
The operator 
\begin{equation}\label{E:Toeplitz}
	T_f\ := \ P M_f P:\, \HH\ \to \HH
\end{equation}
is called the {\em Toeplitz operator defined by $f$}.  We denote by $T_f^\pm$ the restriction of $T_f$ to $\HH^\pm$.  This is a self-adjoint odd symmetric operator which is even with respect to the grading  $\HH=\HH^+\oplus\HH^-$.
\end{definition}

\begin{definition}\label{D:invertible}
We say that a matrix-valued function $f\in C^\infty_g(M;k)$ is {\em invertible at infinity} if there exists a compact set $K\subset M$ and $C_1>0$  such that $f(x)$ is an invertible matrix for all $x\not\in K$ and 
\begin{equation}\label{E:f-1<C}
    \big\|f(x)^{-1}\big\|<C_1 \qquad \text{for all} \quad x\not\in K.
\end{equation}
\end{definition}

The following result is proven in \cite[Lemma~2.6]{Bunke00}
\begin{proposition}\label{P:Fredholm}
If $D$ satisfies Assumption~\ref{A:1} and  $f\in C^\infty_g(M;k)$  is invertible at infinity then the Toeplitz operator $T_f$ is Fredholm. 
\end{proposition}

\subsection{A quaternionic bundle structure}\label{SS:quaternionix_for_Toeplitz}
Assume now that $\theta^M:M\to M$ is an involution. Let $\theta^E:E\to \theta^*E$ be an anti-linear \textit{involution} (not anti-involution!), which preserves the Hermitian metric on $E$. We assume that $\theta^E$ has an odd grading degree, i.e. $\theta^E:E^\pm\to E^\mp$. Let $\theta^{\CC^k}:\CC^k\to \CC^k$ be an anti-unitary anti-involution. By Lemma~\ref{L:even}, it implies that $k$ is even. The tensor product 
\begin{equation}\label{E:thetaECk}
     \theta^{E\otimes\CC^k}\ := \ \theta^E\otimes\theta^{\CC^k}
\end{equation}
is an anti-unitary anti-involution of odd grading degree. As in \eqref{E:tau=} it induces an anti-unitary anti-involution $\tau:\Gamma(M,E\otimes\CC^k)\to \Gamma(M,E\otimes\CC^k)$. 

\begin{assumption}\label{A:diracisodd}
The generalized Dirac operator $D$ on $E\otimes\CC^k$ is graded odd symmetric, $\tau D\tau^{-1}= D$.
\end{assumption}

\subsection{An odd symmetric family of matrix-valued functions}\label{SS:family_of_functions}
Let $S^1= \{e^{it}:t\in [-\pi,\pi]\}$ be the unit circle. Consider the involution $\theta^{S^1}:t\mapsto -t$ on $S^1$. Then $(S^1,\theta^{S^1})$ is an involutive manifold. 

Set $\M= S^1\times M$. We write points of $\M$ as $(t,x)$, $t\in S^1, \ x\in M$, and view a  smooth matrix-valued function $\f:\M\to \Herm(k)$ as a smooth family $f_t:M\to \Herm(k)$, $t\in S^1$.  Thus we usually write this function as $f_t(x)$. We are interested in the functions with $f_t\in C^\infty_g(M;k)$.

The involutions $\theta^M$ and $\theta^{S^1}$ define an involution $\theta^\M:= \theta^{S^1}\times\theta^M$ on $\M$.  As usual, we define the anti-unitary anti-involution $\tau$ on the spaces of matrix-valued functions on $\M$ by 
\begin{equation}\label{E:theta inv matrices}
    (\tau f)_t(x)\ = \ \theta^{\CC^k}\,f_{\theta^{S^1}t}(\theta^M x)\, (\theta^{\CC^k})^{-1}.
\end{equation}

\begin{definition}\label{D:oddsymmetric_function}
A smooth family  $\f:S^1\to C^\infty_g(M;k)$, $t\mapsto f_t$, of matrix-valued functions is called {\em odd symmetric} if  $\tau f=f$.  

We denote by $\F_\tau(M,k)$ the set of smooth families $\f=\{f_t\}$ of odd symmetric matrix-valued functions, such that $f_t$ is  invertible at infinity of $M$ for all $t\in S^1$ and  there exists a constant $C_2>0$ such that 
\begin{equation}\label{E:ddtf<C}
    \big\|\frac{\p}{\p t}f_t(x)\big\|< C_2, \qquad
    \text{for all}\quad \in S^1,\ x\in M. 
\end{equation}
\end{definition}
Suppose $\f= \{f_t\}\in \F_\tau(M;k)$. It follows from \eqref{E:f-1<C} and \eqref{E:ddtf<C} that there exists a compact set $K\subset M$ and  a large enough constant $\alpha>0$ such that   
\begin{equation}\label{E:ddtf<alpha}
	\frac{\p}{\p t}f_t(x)\ < \ 
	\frac{\alpha}2\,f_t(x)^2,
	\qquad\text{for all} \ \ x\not\in K.
\end{equation}
Here the inequality $A<B$ between two self-adjoint matrices means that for any vector $v\not=0$, we have $\<Av,v\><\<Bv,v\>$.

\subsection{The graded half-spectral flow of a family of self-adjoint Toeplitz operators}\label{SS:family}
 If $\f\in \F_\tau(M;k)$ then $T_{\f}$ ($t\in S^1$) is a periodic family of odd symmetric self-adjoint Fredholm operators (the Fredholmness follows from Proposition~\ref{P:Fredholm}). Recall that we denote by $T_{\f}^\pm$ the restriction of $T_\f$ to $\HH^\pm$.
 Our goal is to compute the graded half-spectral flow 
 \[
    \spf_\tau T_\f^+\ - \ \spf_\tau T_\f^-
 \]
 of this family.

\subsection{A Callias-type operator on $S^1\times M$}\label{SS:CalliasSxM} Recall that $\M= S^1\times M$ is an involutive manifold with $\theta^{\MM}=\theta^{S^1}\times \theta^M$. Denote by $\pi_1:S^1\times M\to S^1$ and $\pi_2:S^1\times M\to M$ the natural projections. By a slight abuse of notation, we denote the pull-backs $\pi^*_1dt,\,\pi_2^*dx\in T^*\M$ by  $dt$ and $dx$ respectively. Set  
\[
	\E := \ \pi_2^*E.
\]
Then $\E$ is naturally an ungraded quaternionic Dirac bundle with Clifford action $\c:T^*\M\to \End(\E)$ such that $\c(dx)= c(dx)$ and $\c(dt)$ is given with respect to the decomposition $\E=\pi_2^*E^+\oplus\pi_2^*E^-$ by the matrix
\[
	\c(dt)\ = \ \begin{pmatrix}
	i\cdot\ID&0\\0&-i\cdot\ID
	\end{pmatrix}.
\] 
Let $\D$ be the corresponding Dirac operator. With respect to the decomposition 
\begin{equation}\label{E:L2=L2timesL2}
		L^2(\M,\E\otimes\CC^k)\ = \ 
		L^2(S^1)\otimes L^2(M,E\otimes\CC^k).
\end{equation}
it takes the form 
\begin{equation}\label{E:tilD}
	\D\ = \ \c(dt)\,\frac{\p}{\p t}\otimes1\ + \ 1\otimes D.
\end{equation}
It follows that $\D$ is odd symmetric. We remark that, as opposed to $D$, the operator $\D$ is not graded.  

Let now $\f=\{f_t\}\in \F_\tau(M;k)$.  We view it as a smooth function on $\M$. Let $\MM_\f:L^2(\M,\E\otimes\CC^k)\to L^2(\M,\E\otimes\CC^k)$ denote the multiplication by $\f$.  It is an odd symmetric operator. The commutator 
\[
	[\D,\MM_\f]\ := \ \D\circ\MM_\f-\MM_\f\circ\D
\] 
is a zero-order differential operator, i.e.,  a bundle map $\E\otimes\CC^k\to \E\otimes\CC^k$.


From \eqref{E:ddtf<alpha} and our assumption that $df_t$ vanishes at infinity, we conclude that there exist  constants $c,d>0$ and  a  compact set $\mathcal{K}\subset \M$, called {\em an essential support} of $\B_{c\f}$, such that 
\begin{equation}\label{E:Calliascond}
	\big[\D,c\,\MM_\f\big](t,x) \ <  \ c^2{\MM_\f(t,x)}^2-d,
	\qquad \text{for all}\ \ (t,x)\not\in \mathcal{K}.
\end{equation}
It follows that 
\begin{equation}\label{E:Bft}
	\B_{c\f}\ := \ \D\ + \ i\,c\,\MM_\f
\end{equation}
is a {\em Callias-type operator} in the sense of \cites{Anghel93Callias,Bunke95} (see also \cite[\S2.5]{BrCecchini17}). In particular, it is Fredholm.  It is also odd symmetric. The $\tau$-index of odd symmetric Callias-type operators was studied in  \cite[\S5]{BrSaeedi24index}.

The main result of this section is the following

\begin{theorem}\label{T:sfToeplitz=Callias}
Suppose $\f\in \F_\tau(M;k)$ be a smooth periodic family of invertible at infinity matrix-valued functions. Suppose that Assumptions~\ref{A:1} and \ref{A:diracisodd} are satisfied. Then 
\begin{equation}\label{E:sfToeplitz=Callias}
	\spf_\tau(T_{\f}^+)\ - \  \spf_\tau(T_{\f}^-)\ = \ \ind_\tau{\B_{c\f}}.
\end{equation}
\end{theorem}

\begin{remark}\label{R:gradedsf+-}
Since both sides of \eqref{E:sfToeplitz=Callias} live in $\ZZ_2$, we can replace the minus sign on the left-hand side with the plus sign. We prefer to write the minus sum to stress the analogy with the $\ZZ$-valued case \cite[Theorem~2.1]{Br19Toeplitz} and to emphasize that this expression gives the {\em graded half-spectral flow} of the family $T_{\f}$.
\end{remark}
\begin{remark}\label{R:nonvanishing}
Note that, as opposed to $df$, the differential $d\f$ does not vanish at infinity. Because of this $\B_{c\f}$ does not satisfy the conditions of Corollary~2.7 of \cite{Bunke00} and its index does not vanish in general. 
\end{remark}

\begin{remark}\label{R:vanishing of sfT-}
In many interesting applications, eg. in the generalized Graf-Porta model,  $H^-=\{0\}$. Hence, $\spf(T_{\f}^-)=0$ and \eqref{E:sfToeplitz=Callias} computes $\spf(T_{\f}^+)$.
\end{remark}

Before giving the proof of Theorem~\ref{T:sfToeplitz=Callias} in Section~\ref{S:prsfToeplitz=Callias}, we present some special cases and applications of this theorem.

\section{The even dimensional case: the generalized bulk-edge correspondence}\label{S:evendim}

Suppose that the dimension of $M$ is even. Then the dimension of $\M$ is odd and by the $\ZZ_2$-valued version of the Callias-type index theorem \cite[\S5]{BrSaeedi24index} the $\tau$-index of $\B_{c\f}$ is equal to the index of a certain Dirac operator on a compact hypersurface $\N\subset \M$. It follows from Theorem~\ref{T:sfToeplitz=Callias} that the half-spectral flow of a family of Toeplitz operators on $M$ is equal to the $\tau$-index of a Dirac-type operator on $\N$. We refer to this equality as a generalized bulk-edge correspondence, since, as we explain in the next section, in the special case when $M$ is a unit disk in $\CC$ this reduces to the bulk-edge correspondence for the Graf-Porta model of topological insulators with time-reversal symmetry \cites{GrafPorta13,Hayashi17}.

\subsection{The $\ZZ_2$-valued Callias-type theorem}\label{SS:Callias_theorem}

We recall the $\ZZ_2$-valued version of the Callias-type theorem from \cite{BrSaeedi24index}.

Let $N\subset M$ be a $\theta^M$-invariant hypersurface such that there is an essential support $\mathcal{K}\subset \M$ of $\B_{c\f}$ whose boundary $\p{}\mathcal{K}= \N:=S^1\times{}N$. In particular, the restriction of $\f$ to $\N$ is invertible and satisfies \eqref{E:Calliascond}. Then there are quaternionic vector bundles $\F_{\N\pm}$ over $\N$ such that 
\[
	\M\times\CC^k \ = \ \F_{\N+}\oplus \F_{\N-},
\]
and the restriction of $\f$ to $\F_{\N+}$ (respectively, $\F_{N-}$) is positive definite (respectively, negative definite).  Since $\f$ is a $\tau$-invariant function, the restriction of $\theta^{\CC^k}$ to $\F_{\N\pm}$ defines a quaternionic structure on these vector bundles.

Let $E_N$ denote the restriction of the bundle $E$ to $N$. Then  $\E_\N:=\pi_2^*E_N$ is the restriction of $\E$ to $\N$. It is naturally a  Dirac bundle over $\N$. The restriction of $\theta^\E\otimes\theta^{\CC^k}$ to $\E_\N\otimes\F_{N+}$ defines a structure of a quaternionic bundle on $\E_\N\otimes\F_{N+}$.  The induced Dirac-type operator $\D_\N$ on $\E_\N\otimes \F_{\N+}$ is odd symmetric.

Let $v$ denote the unit normal vector to $\N$ pointing towards $\mathcal{K}$. Then  $i\c(v):\E_\N\to \E_\N$ is an involution. We denote by $\E_\N^{\pm1}$ the eigenspace of $i\c(v)$ with eigenvalue $\pm1$. This defines a grading $\E_\N= \E_\N^{+1}\oplus\E_\N^{-1}$  on the Dirac bundle $\E_N$. 
The operator $\D_\N$ is odd with respect to the induced grading on $\E_\N\otimes \F_{\N+}$.  (This grading is different from the one induced by the grading on $E$. Note that the operator $\D_\N$  is not an odd operator with respect to the grading induced by the grading on $E$).  We denote by $\D_\N^\pm$ the restriction of $\D_\N$ to $\E_\N^{\pm1}\otimes \F_{\N+}$.

Theorem~5.3 of \cite{BrSaeedi24index} states that 
\begin{equation}\label{E:indB=indD}
		\ind_\tau \B_{c\f}\ = \ \ind_\tau \D_{\N}^+.
\end{equation}

\subsection{Computation of the graded spectral flow in even dimensional case}\label{SS:sf=index_even}
We now get the following corollary of Theorem~\ref{T:sfToeplitz=Callias}.
\begin{corollary}\label{C:evencase}
Under the conditions of Theorem~\ref{T:sfToeplitz=Callias} assume that $\dim{}M$ is even. Let $j:N\hookrightarrow M$ be a hypersurface such that there is an essential support $\mathcal{K}\subset \M$ of $\B_{c\f}$ whose boundary $\p{}\mathcal{K}= \N:=S^1\times{}N$. Then 
\begin{equation}\label{E:sfToeplitz=cohom}
	\spf_\tau(T_{\f}^+)\ - \  \spf_\tau(T_{\f}^-)\ = \ \ind_\tau \D_{\N}^+.
\end{equation}

\end{corollary}
\subsection{The pseudo-convex domain}\label{SS:pseudo-convex}
\newcommand{\op}{\bar{\p}}

Let $M$ be a bounded strongly pseudo-convex domain in $\CC^n$ with smooth boundary $N:=\p{}\oM$. Let $g^M$ be the Bergman metric on $M$, cf. 
\cite[\S7]{Stein72book}.  We define $E= \Lambda^{n,*}(T^*M)$  and set $E^+= \Lambda^{n,\even}(T^*M)$, $E^-= \Lambda^{n,\odd}(T^*M)$. Then $E$ is naturally a Dirac bundle over $M$ whose space of smooth section coincides with the Dolbeault 
complex $\Omega^{n,\b}(M)$ of $M$ with coefficients in the canonical bundle $K=\Lambda^{n,0}(T^*M)$. Moreover, the corresponding Dirac operator is given by 
\[
	D\ = \ \op \ + \ \op^*, 
\]
where $\op$ is the Dolbeault differential and $\op^*$ its adjoint with respect to the $L^2$-metric induced by the Bergman metric on $M$. Then $(M,g^M)$ is a complete K\"ahler manifold. 

By \cite[\S5]{DonnellyFefferman83} zero is an isolated point of the spectrum of $D$ and $\HH:=\ker D$ is a subset of the space of $(n,0)$-forms
\begin{equation}\label{E:kernelBergman}
	\HH\ := \ \ker D\ \subset \ \Omega^{n,0}(M).
\end{equation}
In particular, Assumption~\ref{A:1} is satisfied and $\HH^-=\{0\}$. 

Let $\theta^M:M\to M$ be an anti-holomorphic metric preserving involution. We define an anti-unitary involution (not anti-involution!) on $E$ by 
\[
    \theta^E:\, \omega\ \mapsto \ \overline{(\theta^M)^*\omega}, \qquad \omega\in \Lambda^{n,*}(T^*M),
\]
where the bar denotes the complex conjugation and $(\theta^M)^*$ is the pull-back. Note that, since $\theta^M$ is anti-holomorphic,  $(\theta^M)^*$ sends $(n,*)$-forms to $(*,n)$-forms and the complex conjugation sends them back to $(n,*)$-forms.  

Let $\theta^{\CC^k}:\CC^k\to \CC^k$ be an anti-unitary anti-involution and let $\theta^{E\otimes\CC^k}$ be given by \eqref{E:thetaECk}. Then $E\otimes\CC^k$ is a quaternionic bundle and $D\otimes1$ is an odd symmetric operator on it. 
In view of \eqref{E:kernelBergman},  Corollary~\ref{C:evencase} implies the following

\begin{corollary}\label{C:genralized_bulk-edge}
Let $\bar{\f}(t,x)$ be a smooth $\tau$-invariant $\Herm(k)$-valued function on $\N:= S^1\times N$ such that $\bar{\f}(t,x)$ is invertible for all $(t\in S^1$, $x\in N$). Let $\f:\M\to \Herm(k)$ be a smooth $\tau$-invariant extension of $\bar{\f}$ to $\M$ (i.e. $\f|_\N=\bar{\f}$). Then $\f\in \F_\tau(M,k)$ and  
\begin{equation}
    \spf_\tau T_\f\ = \ \ind_\tau\D_\N^+.
\end{equation}
\end{corollary}
We refer to this result as the \textbf{generalized bulk-edge correspondence} for the reason explained in the next section.

\section{The Graf-Porta model of topological insulators with time-reversal symmetry}\label{S:Graf-Porta}

In this section, we briefly review a tight-binding model for two-dimensional topological insulators with fermionic time-reversal symmetry, following the description in \cite{Hayashi17} (see also \cite{GrafPorta13}) and show that the bulk-edge correspondence for this model follows immediately from our Corollary~\ref{C:genralized_bulk-edge}.

\subsection{The bulk Hamiltonian}\label{SS:bulkGrafPorta}
We consider the following  {\em tight binding model}: The  ``bulk" state space is the space $l^2(\ZZ\times\ZZ,\CC^k)$ of square-integrable sequences 
$\phi\ = \ \big\{\phi_{mn}\big\}_{(m,n)\in\ZZ\times\ZZ}$,  where $\phi_{mn}\in \CC^k$.

The {\em bulk Hamiltonian} $H:l^2(\ZZ\times\ZZ,\CC^k)\to l^2(\ZZ\times\ZZ,\CC^k)$ is defined as follows: consider a sequence of matrices $A_{p,q}:\CC^k\to \CC^k$ such that $\sum_{pq}\|A_{p,q}\|<\infty$ and define $H$ by 
\[
    (H\phi)_{mn}\ = \ \sum_{p,q}\, A_{p,q}\,\phi_{m-p,n-q}.
\]
We assume that $H$ is self-adjoint, which is equivalent to the equality 
\begin{equation}\label{E:Apq=Aa-p-q}
        A_{-p,-q}\ = \ A_{p,q}^*.
\end{equation}

The  Fourier transform  of $H$ is a family of self-adjoint $k\times k$-matrices 
\[
    H(t,s) \ := \ \sum_{pq} A_{p,q}\,e^{ipt}\,e^{iqs} \qquad 
    (t,s)\in [-\pi,\pi]\times [-\pi,\pi]\ \simeq \
    S^1\times S^1.
\]
depending smoothly on $s$ and $t$.

{\em We assume that the bulk Hamiltonian has a spectral gap at Fermi level $\mu\in \RR$}, i.e. there exists $\epsilon>0$ such that the spectrum of $H(s,t)$ does not intersect the interval $(\mu-\epsilon,\mu+\epsilon)$ for all $(s,t)\in S^1\times S^1$. In particular, the operator $H(s,t)-\mu$ is invertible  for all $(s,t)\in S^1\times S^1$. Thus the trivial bundle $(S^1\times{}S^1)\times\CC^k$ over the torus $S^1\times{}S^1$ decomposes into the direct sum of subbundles 
\[
	(S^1\times{}S^1)\times\CC^k\ = \ \F_+\oplus \F_-
\]
such that the restriction of $H(s,t)-\mu$ to $\F_+$  is positive definite and the restriction of  $H(s,t)-\mu$ to $\F_-$ is negative define. The bundle $\F_+$ is called the {\em Bloch bundle}.

\subsection{The time-reversal symmetry}\label{SS:time-reversal}
Define the involution $\theta^{S^1\times S^1}: S^1\times S^1\to S^1\times S^1$ by $\theta^{S^1\times S^1}(t,s)= (-t,-s)$ ($(s,t)\in [-\pi,\pi]\times [-\pi,\pi]$). Let 
\[
    \theta^{\CC^k}:\, \CC^k\ \to \ \CC^k
\]
be an anti-unitary anti-involution (hence, $k$ is even). 

The anti-unitary anti-involution  
\[
    \tau:\, \phi(t,s)\ \to \ \theta^{\CC^k}\phi(-t,-s) 
\]
on the space $L^2(S^1\times S^1,\CC^k)$ of square-integrable $\CC^k$-valued functions on $S^1\times S^1$  is called a \textbf{time-reversal symmetry}. We assume that $H(s,t)$ is $\tau$-invariant. Equivalently, this means that 
\begin{equation}\label{E:oddsymmetricApq}
        \theta^{\CC^k}\,A_{p,q}\,(\theta^{\CC^k})^{-1}\ = \ A_{-p,-q}^*\ = \ A_{p,q},
\end{equation}
where in the last equality we used \eqref{E:Apq=Aa-p-q}.

\begin{remark}
In \cite{GrafPorta13} the authors consider the operators $H_q(t):= \sum_p A_{p,q}e^{ipt}$. Then the time-reversal symmetry condition \eqref{E:oddsymmetricApq} becomes 
\[
    \theta^{\CC^k}\,H_q(t)\,(\theta^{\CC^k})^{-1}\ = \ H_q(-t),
\]
which agrees with Definition~2.4 of \cite{GrafPorta13}.
\end{remark}
Notice that  $\tau$ preserves the subbundles $\F_\pm$ of the trivial bundle over $S^1\times S^1$. Hence, $\F_\pm$ a quaternionic bundles over $S^1\times S^1$ (in particular, their ranks are even).

\subsection{The $\ZZ_2$-valued bulk index}\label{SS:bulk_index}
Let $M=\{z\in \CC:\,|z|<1\}$ be the unit disc in $\CC$. It is a pseudo-convex domain with boundary $\p M= S^1$. As usual, we set $\M:= S^1\times M$ so that $\N:= \p\M= S^1\times S^1$. The function 
\begin{equation}\label{E:f(t,z)}
      \f(t,z)\ := \ \sum_{pq} A_{p,q}\,e^{ipt}\,z^q, \qquad 
    t\in [-\pi,\pi], \ z\in M,  
\end{equation}
is an extension of $f(t,s):=H(t,s)$ from $\p\M$ to $\M$ as in Corollary~\ref{C:genralized_bulk-edge}.  Let $T_\f$ and $\D_\N$ be as in Section~\ref{SS:pseudo-convex}. Then $\D_\N$ is an odd symmetric Dirac-type operator acting on the sections of the bundle $\F_+$ over $\p\M= S^1\times S^1$. 

\begin{definition}\label{D:bulkindex}
The {\em $\ZZ_2$-valued bulk index} of the Hamiltonian $H$ is 
\begin{equation}\label{E:bulkindex}
	\Ib\ := \ \ind_\tau \D_\N^+.
\end{equation}
\end{definition} 
Graf and Porta, \cite[\S4]{GrafPorta13}, give a different, more combinatorial, definition of the bulk index. 

\begin{lemma}\label{L:bulk=GrafPorta}
Our definition of the bulk index is equivalent to the Definition~4.8 of \cite{GrafPorta13}. 
\end{lemma}
\begin{proof}
Since the set of Dirac-type operators on the space of differential forms  $\Omega^{n,*}(M,\F_+)$ is connected, the bulk index depends only on the pair $(\F_+,\tau)$. 

More generally, for any quaternionic bundle $\F$ over $\N=S^1\times S^1$, there exists a complementary quaternionic bundle $\F'$ such that $\F\oplus\F'\simeq \N\times \CC^k$ and function $H:\M\to \Herm(k)$ such that $\F=\F_+$ (where $\F_+$ is as in Section~\ref{SS:bulkGrafPorta}). Thus the bulk index can be viewed as a map from the set of quaternionic bundles over $S^1\times S^1$ to $\ZZ^2$. 

The Grothendieck group of quaternionic vector bundles over an involutive manifold $X$ is the symplectic K-theory $KSp^0(X)$, see \cite[\S2]{Hayashi17} and references therein. In particular, 
\[
    KSp^0(S^1\times S^1)\ = \ \ZZ\oplus\ZZ_2,
\]
where the $\ZZ$ summand corresponds to the product bundles $E=(S^1\times S^1)\times \CC^n$, cf. \cite[\S7.1]{Hayashi17}. 

Recall that, by Lemma~\ref{L:even}, the rank of any quaternionic bundle over $S^1\times S^1$ is even. It follows that for a product bundle, the kernel of the Dirac operator on $\Omega^{n,*}(M,\CC^n)= \Omega^{n,*}(M)\otimes\CC^n$ is even-dimensional. Thus the bulk index vanishes on product bundles.  Hence, the bulk index defines a map 
\[
    \Ib:\, KSp^0(S^1\times S^1)/\ZZ\simeq \ZZ_2\ \longrightarrow  \ \ZZ_2.
\]
There is only one non-trivial map like this. Example~3 of \cite{BrSaeedi24index} shows that the bulk index is not always equal to zero and thus defines a unique non-trivial map from $KSp^0(S^1\times S^1)/\ZZ$ to $\ZZ_2$. 

The Graf-Porta construction of the bulk index, \cite[Definition~4.8]{GrafPorta13} also defines a non-trivial map $KSp^0(S^1\times S^1)/\ZZ\to \ZZ_2$ (see also  \cite[\S7.1]{Hayashi17} where the Graf-Porta bulk index is described as a map from $KSp^0(S^1\times S^1)/\ZZ$). Hence, this index must coincide with $\Ib$.
\end{proof}

\subsection{The edge Hamiltonian}\label{SS:edgeGrafPorta}
We now take the boundary into account. The ``edge" state space is the space  $l^2(\ZZ_{\ge0}\times\ZZ,\CC^k)$ of square integrable sequences of vectors in $\CC^k$ on the half-lattice $\ZZ_{\ge0}\times\ZZ$. 

The Fourier transform of the bulk Hamiltonian $H:l^2(\ZZ\times\ZZ,\CC^k)\to l^2(\ZZ\times\ZZ,\CC^k)$ in the direction ``along the edge" 
\begin{equation}\label{E:H(t)}
        H_q(t)\ := \ \sum_{p} A_{p,q}\,e^{ipt}, \qquad 
    t\in [-\pi,\pi], 
\end{equation}
transforms $H$ into a family of self-adjoint translation-invariant operators
\begin{equation*}
	H(t):\,l^2(\ZZ,\CC^k)\ \to \ l^2(\ZZ,\CC^k). \qquad 
    \big(H(t) \phi\big)_n \ := \ \sum_p H_p(t)\,\phi_{n-p}.
\end{equation*}
Then $H(t)$ depends smoothly on $t\in S^1$.  It follows from \eqref{E:Apq=Aa-p-q} that $H(-t)= H(t)^*$ and \eqref{E:oddsymmetricApq} is equivalent to 
\begin{equation}\label{E:H(t)odd}
    \theta^{\CC^k}\, H(t)\, (\theta^{\CC^k})^{-1}\ = \ H(-t).
\end{equation}
(compare to Definition~2.4 of \cite{GrafPorta13}). 

We view $S^1\simeq [-\pi,\pi]$ as and involutive manifold with involution $\theta^{S^1}:t\to -t$. The anti-involution $\theta^{\CC^k}$ defines a structure of a quaternionic bundle on the trivial bundle $S^1\times\CC^k$ over $S^1$. As usual, we define an anti-involution
\[
    \tau:\,L^2\big(S^1,l^2(\ZZ,\CC^k)\big)\ \to \ L^2\big(S^1,l^2(\ZZ,\CC^k)\big), 
    \qquad
    \tau:\,f(t) \ \mapsto \ \theta^{\CC^k}f(-t).
\]
Then $H(t)$ is $\tau$-invariant by \eqref{E:H(t)odd}.

Let $\Pi:l^2(\ZZ,\CC^k)\to l^2(\ZZ_{\ge0},\CC^k)$ denote the projection. Clearly, $\theta^{\CC^k}\, \Pi\, (\theta^{\CC^k})^{-1} =  \Pi$.

\begin{definition}\label{D:edgeHamiltonian}
The {\em edge Hamiltonian} is the family of Toeplitz operators
\begin{equation}\label{E:egdeHamiltonian}
	H^\#(t)\  := \ \Pi\circ H(t)\circ \Pi:\,
	l^2(\ZZ_{\ge0},\CC^k)\ \to \ l^2(\ZZ_{\ge0},\CC^k),
	\qquad t\in S^1.
\end{equation}
\end{definition}
Then 
\begin{equation}\label{E:Hsharp(t)odd}
    \theta^{\CC^k}\, H^\#(t)\, (\theta^{\CC^k})^{-1}\ = \ H^\#(-t).
\end{equation}
\begin{definition}\label{D:edgeindex}
The {\em $\ZZ_2$-valued edge index} $\Ie$ of  the Hamiltonian $H$ is the half-spectral flow of the edge Hamiltonian
\begin{equation}\label{E:edgeindex}
	\Ie\ := \ \spf_\tau\big(H^\#(t)\big). 
\end{equation}
\end{definition}
\begin{theorem}[Bulk-edge correspondence]\label{T:bulkedge}
\(\displaystyle
	\Ib\ = \ \Ie.
\)
\end{theorem}
\begin{proof}
Consider the unit disc $B:= \{z\in \CC:\, |z|\le 1\}$. This is a strongly pseudoconvex domain in $\CC$. We view the bulk Hamiltonian $H(s,t)$ as a function on $\p B\times{}S^1$ with values in the set $\Herm(k)$ of invertible Hermitian $k\times{}k$-matrices. 

We endow $B$ with the Bergman metric and consider the Dolbeault-Dirac  operator $D=\op+\op^*$ on the space $L^2\Omega^{1,\b}(B,\CC^k)=  L^2\Omega^{1,0}(B,\CC^k)\oplus L^2\Omega^{1,1}(B,\CC^k)$ of square-integrable $(1,\b)$-forms on $B$. Let $\HH:= \ker D$. Then $\HH\subset L^2\Omega^{1,0}(B,\CC^k)$, cf.  \cite[\S5]{DonnellyFefferman83}. Let $P:L^2\Omega^{1,\b}(B,\CC^k)\to \HH$ denote the orthogonal projection. 

As in \eqref{E:f(t,z)} we define a family of functions on the unit disc by 
\[
    f_t(z)\ := \  \ \sum_{pq} A_{p,q}\,e^{ipt}\,z^q, \qquad 
    t\in [-\pi,\pi], \ z\in M, 
\]
and consider the family of Toeplitz operators
\[
    T_{f_t}\ :=\ P\circ M_f\circ P:\,\HH\ \to \ \HH, \qquad t\in [-\pi,\pi],
\]
This family is closely related to the family $H^\sharp(t)$ as we shall now explain. 

The map $\varphi:l^2(\ZZ_{\ge0} \to \HH$ which sends 
$u= \{u_j\}_{j\ge0}\in l^2(\ZZ_{\ge0},\CC^k)$ to the 1-form
\[
	\phi(u)\ := \ \sum_{j\ge0}\,u_j\,z^j\,dz\ \in \ \HH, 
\]
is an isomorphism of vector spaces. By \cite{Coburn73}, the difference 
\[
	T_{f_t}\ - \ \phi\circ H^\#(t)\circ \phi^{-1}:\,\HH\ \to \ \HH, 
 \qquad t\in[-\p,\pi].
\] 
is a continuous family of compact operators. By Proposition~\ref{P:sfcrossings tau},
\[
	\spf_\tau(T_{f_t})\ = \ \spf_\tau\big(\, \phi\circ H^\#(t)\circ \phi^{-1}\,\big)
	\ = \ \spf_\tau\big(H^\#(t)\big).
\]
The theorem follows now from definitions of bulk and edge indexes, \eqref{E:bulkindex}, \eqref{E:edgeindex}, and Corollary~\ref{C:genralized_bulk-edge}.
\end{proof}

\section{The proof of Theorem~\ref{T:sfToeplitz=Callias}}\label{S:prsfToeplitz=Callias}

As in \cite[\S5]{Br19Toeplitz}, the proof of Theorem~\ref{T:sfToeplitz=Callias}  consists of two steps. First (Lemma~\ref{L:sf=ind}) we apply Theorem~\ref{T:RScircle} to conclude that the half-spectral flow $\spf_\tau(T_{f_t})$ is equal to the $\tau$-index of a certain odd symmetric operator on $\M$. Then, using the argument similar to \cite{Bunke00} we show that the latter index is equal to the $\tau$-index of $\B_{c\f}$.

\subsection{The spectral flow as an index}\label{SS:sf=index}
Recall that the space $L^2(\M,\E\otimes\CC^k)$ has an involution $\tau$. 
We view the space $L^2(S^1,\HH)$ of square integrable functions with values in $\HH\subset L^2(M,\E\otimes\CC^k)$ as a subspace of $L^2(\M,\E\otimes\CC^k)$. Then $\tau$ induces an anti-involution on $L^2(S^1,\HH)$ which we still denote by $\tau$. 
The family  $T_{f_t}$ naturally induces an operator $L^2(S^1,\HH)\to L^2(S^1,\HH)$ which we still denote by $T_{f_t}$. Let $T_{f_t}^\pm$ denote the restriction of $T_{f_t}$ to $L^2(S^1,\HH^\pm)$. These operators are odd symmetric. 

Theorem~~\ref{T:RScircle} implies the following

\begin{lemma}\label{L:sf=ind}
Under the assumptions of Theorem~\ref{T:sfToeplitz=Callias} we have
\begin{equation}\label{E:sf=ind}
	\spf_\tau(T_{f_t}^\pm)\ = \  
	\ind_\tau\,\big(\frac{\p}{\p t}+T_{f_t}^\pm\big)\Big|_{L^2(S^1,\HH^\pm)}.
\end{equation}
\end{lemma}

Since $\spf_\tau(T_{f_t}^\pm)=-\spf_\tau(-T_{f_t}^\pm)$, the spectral flows $\spf_\tau(T_{f_t}^\pm)$ and $\spf_\tau(-T_{f_t}^\pm)$ are equal modulo 2. Hence, \eqref{E:sf=ind} is equivalent to 
\begin{equation}\label{E:sf=-ind}
	\spf_\tau(T_{f_t}^\pm)\ = \  
	\ind_\tau\,\big(\frac{\p}{\p t}-T_{f_t}^\pm\big)\Big|_{L^2(S^1,\HH^\pm)}.
\end{equation}

\begin{lemma}\label{L:indddt=indBf}
Under the assumptions of Theorem~\ref{T:sfToeplitz=Callias} we have
\begin{equation}\label{E:indddt=indTf}
	 \ind_\tau \B_{c\f}\ = \ 
	 \ind_\tau\,\big(\frac{\p}{\p t}-T_{f_t}^+\big)\Big|_{L^2(S^1,\HH^+)}
	 \ + \ 
	 \ind_\tau\,\big(\frac{\p}{\p t}+T_{f_t}^-\big)\Big|_{L^2(S^1,\HH^-)}.
\end{equation}
\end{lemma}
\begin{proof}
Recall that we denote by $P$ the orthogonal projection $L^2(M,E\otimes\CC^k)\to \HH$. To simplify the notation we write $P$ also for the operator $1\otimes P$ on 
$L^2(\M,\E\otimes\CC^k)$ and set $Q:=1-P$. Then $L^2(S^1,\HH)$ coincides with the image of the projection 
\[
	P:\, L^2(\M,\E\otimes\CC^k)\ \to\  L^2(\M,\E\otimes\CC^k).
\]

Set 
\[
	A\ := \ \c(dt)\,\frac{\p}{\p t}\ + \ i\,c\,\MM_\f.
\]
Then $\B_{c\f}=1\otimes D+A$. Consider a one-parameter family of operators
\[
	\B_{c\f,u}\ := \ 1\otimes D\ + \ u\,A,
	\qquad 0\le u\le 1.
\]
It is shown in the proof of Lemma~5.3 of \cite{Br19Toeplitz} that the operator
\[
	\B_{c\f,u}\ -\ P\circ \B_{c\f,u}\circ P|_{\IM P}\ - \ 
    Q\circ \B_{c\f,u}\circ Q|_{\IM Q}
\]
is compact. Hence, 
\begin{equation}\label{E:indB=diag}
	\ind_\tau \B_{c\f,u}\ = \ 
	\ind_\tau P\circ \B_{c\f,u}\circ P|_{\IM P}\ + \ 
	\ind_\tau Q\circ \B_{c\f,u}\circ Q|_{\IM Q}.
\end{equation}
The operator $Q\circ \B_{c\f,0}\circ Q|_{\IM Q}= Q\circ D\circ Q|_{\IM Q}$ is invertible. Hence, it is Fredholm and its $\tau$-index is equal to 0. We conclude that  
$Q\circ \B_{c\f,u}\circ Q|_{\IM Q}$ ($0\le u\le 1$) is a continuous family of Fredholm odd symmetric operators with 
\[
	\ind_\tau Q\circ \B_{c\f,u}\circ Q|_{\IM Q}\ = \ 0.
\]
From \eqref{E:indB=diag} we now obtain 
\begin{multline}\notag
	\ind_\tau \B_{c\f}\ = \ \ind \B_{c\f,1} \ = \ 
	\ind_\tau P\circ \B_{c\f,1}\circ P|_{\IM P}
	\\ = \ 
	\ind_\tau 
	  P^+\circ\big(i\frac{\p}{\p t}-ic\M_\f\big)\circ P^+\big|_{\IM P^+}
	\ + \ 
	\ind_\tau 
	  P^-\circ\big(-i\frac{\p}{\p t}-ic\M_\f\big)\circ P^-\big|_{\IM P^-}
	\\ = \ 
	\ind_\tau \big(\frac{\p}{\p t}-T_{f_t}^+\big)\big|_{L^2(S^1,\HH^+)}
	\ + \ 
	\ind_\tau \big(\frac{\p}{\p t}+T_{f_t}^-\big)\big|_{L^2(S^1,\HH^-)}.
\end{multline}
\end{proof}

\subsection{Proof of Theorem~\ref{T:sfToeplitz=Callias}}\label{SS:prsfToeplitz=Callias}
Theorem~\ref{T:sfToeplitz=Callias} follows now from \eqref{E:sf=ind}, \eqref{E:sf=-ind}, and \eqref{E:indddt=indTf}.\hfill$\square$

\section*{\textbf{Conflict of interest}}
The authors have no relevant financial or non-financial interests to disclose.

\section*{\textbf{Data Availability}}

Data sharing is not applicable to this article as no new data were created or analyzed in this study.


\begin{bibdiv}
\begin{biblist}


\bib{Anghel93Callias}{article}{
      author={Anghel, Nicolae},
       title={On the index of {C}allias-type operators},
        date={1993},
        ISSN={1016-443X},
     journal={Geom. Funct. Anal.},
      volume={3},
      number={5},
       pages={431\ndash 438},
         url={http://dx.doi.org/10.1007/BF01896237},
}

\bib{APS3}{article}{
      author={Atiyah, M.~F.},
      author={Patodi, V.~K.},
      author={Singer, I.~M.},
       title={Spectral asymmetry and {R}iemannian geometry. {III}},
        date={1976},
     journal={Math. Proc. Cambridge Philos. Soc.},
      volume={79},
      number={1},
       pages={71\ndash 99},
}

\bib{AtSinger69}{article}{
      author={Atiyah, M.~F.},
      author={Singer, I.~M.},
       title={Index theory for skew-adjoint {F}redholm operators},
        date={1969},
        ISSN={0073-8301},
     journal={Inst. Hautes \'Etudes Sci. Publ. Math.},
      number={37},
       pages={5\ndash 26},
         url={http://www.numdam.org/item?id=PMIHES_1969__37__5_0},
      review={\MR{0285033}},
}

\bib{BandaraGoffengRozen25}{article}{
      author={Bandara, Lashi},
      author={Goffeng, Magnus},
      author={Andreas, Rozen},
       title={Index theory for first-order operators on manifolds with
  {L}ipschitz boundary},
     journal={To appear},
}




\bib{BarBallmann12}{incollection}{
      author={B\"ar, Christian},
      author={Ballmann, Werner},
       title={Boundary value problems for elliptic differential operators of
  first order},
        date={2012},
   booktitle={Surveys in differential geometry. {V}ol. {XVII}},
      series={Surv. Differ. Geom.},
      volume={17},
   publisher={Int. Press, Boston, MA},
       pages={1\ndash 78},
         url={http://dx.doi.org/10.4310/SDG.2012.v17.n1.a1},
}

\bib{BourneCareyLeschRennie22}{article}{
      author={Bourne, Chris},
      author={Carey, Alan~L},
      author={Lesch, Matthias},
      author={Rennie, Adam},
       title={The {K}{O}-valued spectral flow for skew-adjoint fredholm
  operators},
        date={2022},
        ISSN={1793-5253},
     journal={Journal of Topology and Analysis},
      volume={14},
      number={2},
       pages={505\ndash 556},
}



\bib{Br19Toeplitz}{article}{
      author={Braverman, Maxim},
       title={Spectral flows of {T}oeplitz operators and bulk-edge
  correspondence},
        date={2019},
        ISSN={1573-0530},
     journal={Letters in Mathematical Physics},
      volume={109},
      number={10},
       pages={2271\ndash 2289},
         url={https://doi.org/10.1007/s11005-019-01187-7},
}

\bib{BrCecchini17}{article}{
      author={Braverman, Maxim},
      author={Cecchini, Simone},
       title={Callias-type operators in von {N}eumann algebras},
        date={2018},
        ISSN={1559-002X},
     journal={The Journal of Geometric Analysis},
      volume={28},
      number={1},
       pages={546\ndash 586},
         url={https://doi.org/10.1007/s12220-017-9832-1},
}


\bib{BrSaeedi24index}{article}{
	author = {Braverman, Maxim},
    author={Sadegh, Ahmad Reza Haj Saeedi},
	journal = {Annales math{\'e}matiques du Qu{\'e}bec, \ DOI: 10.1007/s40316-024-00228-5},
	pages = {arXiv:2403.13999},
	title = {On the $\mathbb{Z}_2$-valued index of elliptic odd symmetric operators on non-compact manifolds},
	year = {2024}}



\bib{BrSaeediYan25}{article}{
	author = {Braverman, Maxim},
    author={Sadegh, Ahmad Reza Haj~Saeedi},
    author={Yan, Junrong},
	journal = {In preparation.},
	title = {A gluing formula for the $\mathbb{Z}_2$-valued index if odd symmetric operators}}


\bib{BrShi21odd}{article}{
	author = {Braverman, Maxim},
        author={Shi, Pengshuai},
	journal = {J. Geom. Anal.},
	number = {4},
	pages = {3713--3763},
	title = {The {A}tiyah-{P}atodi-{S}inger index on manifolds with non-compact boundary},
	volume = {31},
	year = {2021}}



\bib{Bunke95}{article}{
      author={Bunke, Ulrich},
       title={A {$K$}-theoretic relative index theorem and {C}allias-type
  {D}irac operators},
        date={1995},
        ISSN={0025-5831},
     journal={Math. Ann.},
      volume={303},
      number={2},
       pages={241\ndash 279},
         url={http://dx.doi.org/10.1007/BF01460989},
      review={\MR{1348799 (96e:58148)}},
}

\bib{Bunke00}{incollection}{
      author={Bunke, Ulrich},
       title={On the index of equivariant {T}oeplitz operators},
        date={2000},
   booktitle={Lie theory and its applications in physics, {III} ({C}lausthal,
  1999)},
   publisher={World Sci. Publ., River Edge, NJ},
       pages={176\ndash 184},
      review={\MR{1888382}},
}

\bib{CareyPhillipsSB19}{article}{
      author={Carey, Alan~L.},
      author={Phillips, John},
      author={Schulz-Baldes, Hermann},
       title={Spectral flow for skew-adjoint {F}redholm operators},
        date={2019},
        ISSN={1664-039X,1664-0403},
     journal={J. Spectr. Theory},
      volume={9},
      number={1},
       pages={137\ndash 170},
         url={https://doi.org/10.4171/JST/243},
      review={\MR{3900782}},
}

\bib{Coburn73}{article}{
      author={Coburn, L.~A.},
       title={Singular integral operators and {T}oeplitz operators on odd
  spheres},
        date={1973/74},
        ISSN={0022-2518},
     journal={Indiana Univ. Math. J.},
      volume={23},
       pages={433\ndash 439},
         url={https://doi.org/10.1512/iumj.1973.23.23036},
      review={\MR{0322595}},
}

\bib{DeNittisGomi15}{article}{
      author={De~Nittis, Giuseppe},
      author={Gomi, Kiyonori},
       title={Classification of ``quaternionic" {B}loch-bundles: topological
  quantum systems of type {A}{I}{I}},
        date={2015},
        ISSN={0010-3616,1432-0916},
     journal={Comm. Math. Phys.},
      volume={339},
      number={1},
       pages={1\ndash 55},
         url={https://doi.org/10.1007/s00220-015-2390-0},
      review={\MR{3366050}},
}

\bib{DeNittisSB15}{article}{
      author={De~Nittis, Giuseppe},
      author={Schulz-Baldes, Hermann},
       title={Spectral flows of dilations of {F}redholm operators},
        date={2015},
        ISSN={0008-4395,1496-4287},
     journal={Canad. Math. Bull.},
      volume={58},
      number={1},
       pages={51\ndash 68},
         url={https://doi.org/10.4153/CMB-2014-055-3},
      review={\MR{3303207}},
}

\bib{DollSB21}{article}{
      author={Doll, Nora},
      author={Schulz-Baldes, Hermann},
       title={Skew localizer and {$\mathbb{Z}_2$}-flows for real index
  pairings},
        date={2021},
        ISSN={0001-8708,1090-2082},
     journal={Adv. Math.},
      volume={392},
       pages={Paper No. 108038, 42},
         url={https://doi.org/10.1016/j.aim.2021.108038},
      review={\MR{4322160}},
}

\bib{DonnellyFefferman83}{article}{
      author={Donnelly, H.},
      author={Fefferman, C.},
       title={{$L^{2}$}-cohomology and index theorem for the {B}ergman metric},
        date={1983},
        ISSN={0003-486X},
     journal={Ann. of Math. (2)},
      volume={118},
      number={3},
       pages={593\ndash 618},
         url={https://doi.org/10.2307/2006983},
}

\bib{Dupont69}{article}{
      author={Dupont, Johan~L.},
       title={Symplectic bundles and {$KR$}-theory},
        date={1969},
        ISSN={0025-5521,1903-1807},
     journal={Math. Scand.},
      volume={24},
       pages={27\ndash 30},
         url={https://doi.org/10.7146/math.scand.a-10918},
      review={\MR{254839}},
}

\bib{GrafPorta13}{article}{
      author={Graf, G.~M.},
      author={Porta, M.},
       title={Bulk-edge correspondence for two-dimensional topological
  insulators},
        date={2013},
        ISSN={0010-3616},
     journal={Comm. Math. Phys.},
      volume={324},
      number={3},
       pages={851\ndash 895},
         url={http://dx.doi.org/10.1007/s00220-013-1819-6},
      review={\MR{3123539}},
}

\bib{Hayashi17}{article}{
      author={Hayashi, Shin},
       title={Bulk-edge correspondence and the cobordism invariance of the
  index},
        date={2017},
     journal={Reviews in Mathematical Physics},
      volume={29},
      number={10},
       pages={1750033},
      eprint={https://doi.org/10.1142/S0129055X17500337},
         url={https://doi.org/10.1142/S0129055X17500337},
}

\bib{KleinMartin1952}{article}{
      author={Klein, Martin~J.},
       title={On a degeneracy theorem of {K}ramers},
    language={eng},
        date={1952},
        ISSN={0002-9505},
     journal={American journal of physics},
      volume={20},
      number={2},
       pages={65\ndash 71},
}

\bib{LawMic89}{book}{
      author={Lawson, H.~B.},
      author={Michelsohn, M.-L.},
       title={Spin geometry},
   publisher={Princeton University Press},
     address={Princeton, New Jersey},
        date={1989},
}

\bib{RobbinSalamon95}{article}{
      author={Robbin, Joel},
      author={Salamon, Dietmar},
       title={The spectral flow and the {M}aslov index},
        date={1995},
        ISSN={0024-6093},
     journal={Bull. London Math. Soc.},
      volume={27},
      number={1},
       pages={1\ndash 33},
         url={https://doi.org/10.1112/blms/27.1.1},
}

\bib{Schulz-Baldes15}{article}{
      author={Schulz-Baldes, Hermann},
       title={{$\mathbb Z_2$}-indices and factorization properties of odd
  symmetric {F}redholm operators},
        date={2015},
        ISSN={1431-0635,1431-0643},
     journal={Doc. Math.},
      volume={20},
       pages={1481\ndash 1500},
         url={https://doi.org/10.3934/dcdsb.2015.20.1031},
}

\bib{Stein72book}{book}{
      author={Stein, E.~M.},
       title={Boundary behavior of holomorphic functions of several complex
  variables},
   publisher={Princeton University Press, Princeton, N.J.; University of Tokyo
  Press, Tokyo},
        date={1972},
        note={Mathematical Notes, No. 11},
      review={\MR{0473215}},
}

\bib{DungenRonge21}{article}{
      author={van~den Dungen, Koen},
      author={Ronge, Lennart},
       title={The {APS}-index and the spectral flow},
        date={2021},
        ISSN={1846-3886,1848-9974},
     journal={Oper. Matrices},
      volume={15},
      number={4},
       pages={1393\ndash 1416},
         url={https://doi.org/10.7153/oam-2021-15-87},
      review={\MR{4364605}},
}

\end{biblist}
\end{bibdiv}
\end{document}